\documentclass[reqno, 11pt]{amsart}
\usepackage[utf8]{inputenc}
\usepackage[english]{babel}
\usepackage{amssymb,amsmath,amsthm}
\usepackage{graphicx,geometry,caption,subcaption}
\usepackage{float}
\usepackage{faktor}
\usepackage{enumerate}
\usepackage{multicol}
\usepackage{tikz}
\usetikzlibrary{patterns}
\usepackage{pgfplots}
\usepackage[all]{xy}
\usepackage{tikz-cd}
\usepackage{comment}
\usepackage{cite}
\usepackage{setspace}
\usepackage{amssymb,amstext}
\usepackage{xcolor}
\pgfplotsset{compat=1.10}
\usepackage{multirow}
\usepackage{xcolor}
\usepackage{ifthen}
\usepackage{array}   

\numberwithin{equation}{section}
\newcommand{\extp}{\@ifnextchar^\@extp{\@extp^{\,}}}
\def\extp^#1{\mathop{\bigwedge\nolimits^{\!#1}}}
\makeatother

\theoremstyle{plain}
\newtheorem{theorem}{Theorem}[section]
\newtheorem{proposition}{Proposition}[section]
\newtheorem{corollary}{Corollary}[section]
\newtheorem{lemma}{Lemma}[section]

\newcolumntype{L}{>{$}l<{$}} 
\newcolumntype{C}{>{$}c<{$}} 

\makeatletter 

\newcommand{\tree}[1]{%
  \def\mylist{#1}%
  \newcount\countitems \countitems=0
  \@for\x:=\mylist\do{\advance\countitems by 1\relax}%
  \ifnum\countitems=3
  \def\first{}\def\second{}\def\third{}%
  \@for\x:=#1\do{%
    \ifx\first\empty
      \edef\first{\x}%
    \else\ifx\second\empty
      \edef\second{\x}%
    \else
      \edef\third{\x}%
    \fi\fi
  }%
  \begin{tikzpicture}[baseline=-3pt, scale=.6]
    \draw[thick] (0,0) -- (90:0.3cm);
    \draw[thick] (0,0) -- (-50:0.6cm);
    \draw[thick] (0,0) -- (230:0.6cm);

    \node[above] at (-0.45,-1.3) {$\first$};
    \node[above] at (0,0.25) {$\second$};
    \node[above] at (0.45,-1.3) {$\third$};
  \end{tikzpicture}%
  \fi
\ifnum\countitems=4
  \def\first{}\def\second{}\def\third{}\def\fourth{}%
  \@for\x:=#1\do{%
    \ifx\first\empty
      \edef\first{\x}%
    \else\ifx\second\empty
      \edef\second{\x}%
    \else\ifx\third\empty
      \edef\third{\x}%
    \else
      \edef\fourth{\x}%
    \fi\fi\fi
  }%
\begin{tikzpicture}[baseline=-3pt, scale=.6]
\draw[thick] (-1,0) -- (0.5,0);
\draw[thick] (-1,0.3) -- (-1, -0.3);
\draw[thick] (0.5,0.3) -- (0.5, -0.3);

\node[above] at (-1,-1.2) {$\first$};
\node[above] at (-1,0.3) {$\second$};
\node[above] at (0.5,0.3) {$\third$};
\node[above] at (0.5,-1.2) {$\fourth$};
\end{tikzpicture}
\fi
 \ifnum\countitems=5
    \def\first{}\def\second{}\def\third{}\def\fourth{}\def\fifth{}%
    \@for\x:=#1\do{%
      \ifx\first\empty \edef\first{\x}%
      \else\ifx\second\empty \edef\second{\x}%
      \else\ifx\third\empty \edef\third{\x}%
      \else\ifx\fourth\empty \edef\fourth{\x}%
      \else \edef\fifth{\x}%
      \fi\fi\fi\fi
    }%
\begin{tikzpicture}[baseline=-3pt, scale=.6]
\draw[thick] (-1,0) -- (1,0);
\draw[thick] (-1,0.3) -- (-1, -0.3);
\draw[thick] (1,0.3) -- (1, -0.3);
\draw[thick] (0,0) -- (0, 0.3);

\node[above] at (-1,-1.2) {$\first$};
\node[above] at (-1,0.3) {$\second$};
\node[above] at (0,0.3) {$\third$};
\node[above] at (1,0.3) {$\fourth$};
\node[above] at (1,-1.2) {$\fifth$};
\end{tikzpicture}
 \fi
}

\makeatother

\begin{document}

\title[A Parametrization of $\mathbb{Z}$-homology $3$-spheres by the 4th Johnson subgroup]{A Parametrization of integral homology $3$-spheres by the fourth Johnson subgroup}


\author{Ricard Riba}
\address{Universitat de Girona, Departament d'Informàtica, Matemàtica aplicada i estadística, Girona, Spain}
\email{ricard.riba@udg.edu}
\thanks{This work was partially supported by the grant PID2024-157757NB-I00 funded by MICIU/AEI/ 10.13039/501100011033 and by ERDF/EU}



\date{\today}


\begin{abstract}
By results of Morita, Pitsch and, more recently, Faes, it is known that any integral homology 3-sphere can be constructed as a Heegaard splitting with a gluing map an element of the fourth Johnson subgroup.
In this work we prove that the equivalence relation on the fourth Johnson subgroup induced by this construction admits an intrinsic description in terms of the fourth Johnson handlebody subgroups.
In addition,
we give an ``antisymmetic'' Lagrangian trace map inspired in the Lagrangian trace map introduced by Faes and compute the image of the third Johnson handlebody subgroups by the third Johnson homomorphism.
\end{abstract}

\maketitle

\section{Introduction}

Let $\Sigma_{g,1}$ be an oriented surface of genus $g$ with a marked $2$-disk and $\mathcal{M}_{g,1}$ its mapping class group relative to the marked disk.
By elementary Morse theory, any oriented $3$-manifold can be obtained by cutting the oriented 3-sphere $\mathbb{S}^3$ along a standardly embedded copy of $\Sigma_{g,1}$ into two handlebodies $\mathcal{H}_g$, $-\mathcal{H}_g$ and gluing them back by an element of the mapping class group $\mathcal{M}_{g,1}$,
that is, as a Heegaard splitting $\mathcal{H}_g\cup_\varphi -\mathcal{H}_g$ with $\varphi \in \mathcal{M}_{g,1}$.
We call \textit{Heegaard map} the map that sends an element $\varphi\in\mathcal{M}_{g,1}$
to the diffeomorphism class of the corresponding Heegaard splitting $\mathcal{H}_g\cup_\varphi -\mathcal{H}_g$.
The lack of the injectivity in this construction is controlled by the inclusions
$\mathcal{M}_{g,1}\hookrightarrow \mathcal{M}_{g+1,1}$, induced by the embeddings
$\Sigma_{g,1}\hookrightarrow \Sigma_{g+1,1}$, and the subgroups
$\mathcal{A}_{g,1}$ and $\mathcal{B}_{g,1}$, formed by the elements of $\mathcal{M}_{g,1}$ that extend respectively to the inner handlebody $\mathcal{H}_g$ and the outer handlebody $-\mathcal{H}_g$, (see Singer \cite[Thm. 12]{singer}). Moreover, by a result of Waldhausen \cite{wald},
the intersection $\mathcal{A}_{g,1}\cap\mathcal{B}_{g,1}$ coincides with the subgroup
$\mathcal{AB}_{g,1}$, formed by the elements of the mapping class group $\mathcal{M}_{g,1}$ that extend to the whole $3$-sphere $\mathbb{S}^3$.

By results of Morita \cite[Thm. 2.2]{mor}, the restriction of the Heegaard map to the Torelli group $\mathcal{T}_{g,1}$, which is the group formed by the elements of $\mathcal{M}_{g,1}$ that act trivially on the first homology group of $\Sigma_{g,1}$,
assembles the whole set of integral homology 3-spheres $\mathcal{S}_{\mathbb{Z}}^3$.

In this work we consider two filtrations of the Torelli group $\mathcal{T}_{g,1}$:
The \textit{Johnson filtration}, which is given by the subgroups $\mathcal{M}_{g,1}(k)$ formed by the elements of $\mathcal{M}_{g,1}$ that act trivially on $\pi_1(\Sigma_{g,1})/\Gamma_k(\pi_1(\Sigma_{g,1}))$, where $\Gamma_k(\pi_1(\Sigma_{g,1}))$
denotes the $k$-th commutator of $\pi_1(\Sigma_{g,1})$.
In particular, $\mathcal{M}_{g,1}(1)=\mathcal{T}_{g,1}$ and by results of Johnson \cite[Thm. 5]{jon_2}, the subgroup $\mathcal{M}_{g,1}(2)$ is the group generated by the Dehn twists along bounding simple closed curves on $\Sigma_{g,1}$.
The \textit{Torelli filtration}, which is given by the lower central series of the Torelli group $\mathcal{T}_{g,1}$, that is, it is formed by the subgroups $\mathcal{T}_{g,1}(k)$ defined inductively by $\mathcal{T}_{g,1}(1)=\mathcal{T}_{g,1}$ and $\mathcal{T}_{g,1}(k+1)=[\mathcal{T}_{g,1},\mathcal{T}_{g,1}(k)]$.

\pagebreak

These two filtrations induce, by Heegaard map, two filtrations on the set $\mathcal{S}_{\mathbb{Z}}^3$. The filtration of $\mathcal{S}_{\mathbb{Z}}^3$ induced by the Torelli filtration is related with the theory of finite type invariants introduced by Ohtsuki in \cite{ohtsuki1}. In particular, by results of Garufalidis and Levine \cite{levine3}, it is known that any finite type invariant of order $k$ vanishes on the subset of $\mathcal{S}_{\mathbb{Z}}^3$ constructed from $\mathcal{T}_{g,1}(2k+1)$. For instance, the Casson invariant, which is a finite type $1$ invariant, vanishes on $\mathcal{T}_{g,1}(3)$.
The filtration of $\mathcal{S}_{\mathbb{Z}}^3$ induced by the Johnson filtration instead, has been less studied.
In particular, based on the results of  Morita \cite[Prop. 2.3]{mor} for $k=2$, Pitsch \cite[Main Thm.]{pitsch3} for $k=3$ and Faes \cite[Thm. B]{faes2} for $k=4$, we know that there is a bijective map:
\begin{equation}
\label{eq:bij-k-int}
\lim_{g\to \infty}\mathcal{A}_{g,1}\backslash\mathcal{M}_{g,1}(k)/\mathcal{B}_{g,1}  \longrightarrow \mathcal{S}_\mathbb{Z}^3.
\end{equation}

But in \cite[Thm. 1.4]{PR} we showed that this map is no longer surjective for $k=5$. This is due to the existence of a finite type $2$ invariant that vanishes on $\mathcal{M}_{g,1}(5)$, which takes the form $\lambda_2-18\lambda^2+3\lambda$, where $\lambda$ denotes the Casson invariant and $\lambda_2$ the second Ohtsuki invariant.

In fact, by results of Pitsch \cite[Thm. 1]{pitsch} for $k=1$ and Faes \cite[Prop 6.7]{faes1} for $k=2,3$,
if we denote 
$ \mathcal{A}_{g,1}(k)=\mathcal{M}_{g,1}(k)\cap \mathcal{A}_{g,1}$ and
$\mathcal{B}_{g,1}(k)=\mathcal{M}_{g,1}(k)\cap \mathcal{B}_{g,1}$,
the equivalence relation on $\mathcal{M}_{g,1}(k)$ given in \eqref{eq:bij-k-int} for $k=1,2,3$ can be reformulated as follows:
\begin{equation}
\label{eq:ref-equiv-int}
\displaystyle{\lim_{g \to \infty}} (\mathcal{A}_{g,1}(k) \setminus \mathcal{M}_{g,1}(k)/\mathcal{B}_{g,1}(k))_{\mathcal{AB}_{g,1}} \longrightarrow \mathcal{S}_\mathbb{Z}^3.
\end{equation}
More precisely, two maps $\phi,\psi \in \mathcal{M}_{g,1}(k)$ are equivalent if and only if there exist maps $\xi_a \in \mathcal{A}_{g,1}(k)$, $\xi_b \in \mathcal{B}_{g,1}(k)$ and $\mu \in \mathcal{AB}_{g,1}$ such that
$ \phi = \mu \xi_a \psi \xi_b \mu^{-1}$.
Using this result, for $k=1,2,3$, given any normalized invariant of homology spheres $F:\mathcal{S}_\mathbb{Z}^3\rightarrow \mathbb{Z}$, by precomposing this invariant with \eqref{eq:ref-equiv-int} we get a family of functions $(F_g)_g$
that satisfies some properties (cf. \cite[Sec. 3]{PR}).
Moreover, the defect of these maps to be homomorphisms induces a family of $2$-cocycles:
 \begin{align*}
 	C_g: \mathcal{M}_{g,1}(k)\times \mathcal{M}_{g,1}(k) & \longrightarrow \mathbb{Z}, \\
 	(\phi,\psi) & \longmapsto F_g(\phi)+F_g(\psi)-F_g(\phi\psi)
 \end{align*}
with inherited properties from $F_g$.

This framework, as it was shown in \cite{pitsch} and \cite{PR}, allow us to study invariants of integral homology $3$-spheres from an algebraic point of view and to provide ``surgery formulas'' for these invariants, which are given by their associated 2-cocycles.
These surgery formulas turns out to become simpler when restricted to deeper levels of the Johnson filtration.
For instance, by \cite{pitsch}, the $2$-cocycle of the Casson invariant is a pullback of a bilinear form on $\Lambda^3 H_1(\Sigma_{g,1},\mathbb{Z})$ by the first Johnson homomorphism, but when restricted to $\mathcal{M}_{g,1}(2)$ such $2$-cocycle is zero because the Casson invariant is a homomorphism on $\mathcal{M}_{g,1}(2)$.
Therefore, it is of our interest to know whether the bijection \eqref{eq:ref-equiv-int} holds true for the fourth Johnson subgroup, which is the last level of the Johnson filtration for which the Heegaard map \eqref{eq:bij-k-int} still assemble the whole set of integral homology $3$-spheres.

The objective of this work is to show that the reformulation \eqref{eq:ref-equiv-int} also holds for $k=4$. To prove this result we first compute the image of the third Johnson homomorphism $\tau_3$, on the third Johnson subgroup $\mathcal{M}_{g,1}(3)$ and its handlebody subgroups
$\mathcal{A}_{g,1}(3)$, $\mathcal{B}_{g,1}(3)$ and $\mathcal{AB}_{g,1}(3)$.
By results of Faes \cite[Cor. 4.6]{FMS}, we know that
$ \tau_3(\mathcal{M}_{g,1}(3))=\tau_3(\mathcal{T}_{g,1}(3))$.
For the handlebody subgroups, if we denote $\mathcal{TA}_{g,1}(k)$, $\mathcal{TB}_{g,1}(k)$ and $\mathcal{TAB}_{g,1}(k)$ the $k$-th term of the lower central series of $\mathcal{TA}_{g,1}$, $\mathcal{TB}_{g,1}$ and $\mathcal{TAB}_{g,1}$ respectively, there are inclusions:
\[
\begin{array}{c}
\tau_3(\mathcal{A}_{g,1}(3))\supset\tau_3(\mathcal{TA}_{g,1}(3)), \qquad
\tau_3(\mathcal{B}_{g,1}(3))\supset\tau_3(\mathcal{TB}_{g,1}(3)), \\[1ex]
\tau_3(\mathcal{AB}_{g,1}(3))\supset\tau_3(\mathcal{TAB}_{g,1}(3)).
\end{array}
\]
We show that in fact, these inclusions are equalities. For this we construct two ``antisymmetric'' Lagrangian trace maps
$ Tr^A_\Lambda$ and $Tr^B_\Lambda$ inspired in the Lagrangian trace maps introduced by Faes in \cite[Sec. 4]{faes1}, and show that they vanish on $\tau_3(\mathcal{A}_{g,1}(3))$ and $\tau_3(\mathcal{B}_{g,1}(3))$ respectively.
Then we prove that the modules $Ker(Tr^A_\Lambda)\cap Im(\tau_3)$ and $Ker(Tr^B_\Lambda)\cap Im(\tau_3)$ are respectively contained in the modules
$\tau_3(\mathcal{TA}_{g,1}(3))$ and $\tau_3(\mathcal{TB}_{g,1}(3))$.
Therefore we get the following result:

\begin{proposition} \label{prop:intro-3}
For any given integer $g\geq 6$, we have the following equalities:
\[
\begin{aligned}
\tau_3(\mathcal{A}_{g,1}(3)) & =\tau_3(\mathcal{TA}_{g,1}(3))=Ker(Tr^A_\Lambda)\cap Im(\tau_3), \\
\tau_3(\mathcal{B}_{g,1}(3)) & =\tau_3(\mathcal{TB}_{g,1}(3))=Ker(Tr^B_\Lambda)\cap Im(\tau_3), \\
\tau_3(\mathcal{AB}_{g,1}(3))& =\tau_3(\mathcal{TAB}_{g,1}(3))=Ker(Tr^A_\Lambda)\cap Ker(Tr^B_\Lambda)\cap Im(\tau_3).
\end{aligned}
\]
\end{proposition}

From this result we get a ``handlebody version'' of \cite[Cor. 4.6]{FMS},

\begin{corollary}
\label{cor:intro-1}
For any given integer $g\geq 6$, we have the following equalities:
\[
\begin{array}{c}
\mathcal{A}_{g,1}(3)=\mathcal{TA}_{g,1}(3)\cdot \mathcal{A}_{g,1}(4), \qquad
\mathcal{B}_{g,1}(3)=\mathcal{TB}_{g,1}(3)\cdot \mathcal{B}_{g,1}(4), \\[1ex]
\mathcal{AB}_{g,1}(3)=\mathcal{TAB}_{g,1}(3)\cdot \mathcal{AB}_{g,1}(4).
\end{array}
\]
\end{corollary}

Moreover, as a direct consequence of Proposition \ref{prop:intro-3}, we also compute the intersection of the images of $\mathcal{A}_{g,1}(3)$ and $\mathcal{B}_{g,1}(3)$ by the third Johnson homomorphism.

\begin{corollary}
\label{cor:intro-2}
For any given integer $g\geq 6$, we have the following equality:
\[
\tau_3(\mathcal{A}_{g,1}(3))\cap \tau_3(\mathcal{B}_{g,1}(3))=\tau_3(\mathcal{AB}_{g,1}(3)).
\]
\end{corollary}

Finally, we prove the core result of this work:

\begin{theorem} \label{thm:equiv-rel-4}
The Heegaard map induces a bijection:
\[
\displaystyle{\lim_{g \to \infty}} (\mathcal{A}_{g,1}(4) \setminus \mathcal{M}_{g,1}(4)/\mathcal{B}_{g,1}(4))_{\mathcal{AB}_{g,1}}  \longrightarrow  \mathcal{S}_\mathbb{Z}^3.
\]
\end{theorem}

{\bf Plan of this paper.} In Section 2 we give some definitions and preliminary results about the Symplectic representation, the Tree Lie algebra, the Johnson homomorphism and coinvariants in tensors of the defining representation of $GL_g(\mathbb{Z})$.
Then, in Section 3 we show that the image of the third Johnson homomorphism can be decomposed as a direct sum of submodules formed by trees with the same coloring of leaves. In addition we also give a reduced set of generators of these submodules.
Later, in Section 4 we define the antisymmectic Lagrangian trace map $Tr_\Lambda^A$ and compute the image of the handlebody subgroups by the third Johnson homomorphims to prove Proposition \ref{prop:intro-3} and deduce Corollaries \ref{cor:intro-1} and \ref{cor:intro-2}.
Finally, we end this section by giving a proof of Theorem \ref{thm:equiv-rel-4}.

{\bf Acknowledgements.} The author would like to thank Prof. Wolfgang Pitsch for all his valuable comments and corrections throughout the development of this work.

\section{Preliminary results}

\subsection{The Symplectic representation}\label{subsec:symprep}
Let $\Sigma_{g,1}$ be an oriented surface of genus $g$ with a marked $2$-disk .
The homology classes of the curves $\{ a_i \ | \ 1 \leq i \leq g\}$ and $\{ b_i \ | \ 1 \leq i \leq g \}$ depicted in Figure~\ref{fig:homology_basis} form a basis of $H_1(\Sigma_{g,1};\mathbb{Z})$.
In addition, if we consider the symplectic form $\omega$ on $H_1(\Sigma_{g,1};\mathbb{Z})$ induced by the transverse intersection of oriented paths on $\Sigma_{g,1}$, we get that the aforementioned basis indeed forms a symplectic basis.
In other words, $\omega(a_i,b_j)=-\omega(b_j,a_i)$ and
$\omega(a_i,b_j)=1$ if and only if $i=j$.

In particular such symplectic basis induces two supplementary transverse Lagrangians $A$ and $B$, respectively generated by the curves $\{ a_i \ | \ 1 \leq i \leq g\}$ and $\{ b_i \ | \ 1 \leq i \leq g \}$. Therefore, as a symplectic space, we have a decomposition $H_1(\Sigma_{g,1};\mathbb{Z}) = A \oplus B$.

\begin{figure}[H]
\begin{center}
\begin{tikzpicture}[scale=.7]
\draw[very thick] (-4.5,-2) -- (5,-2);
\draw[very thick] (-4.5,2) -- (5,2);
\draw[very thick] (-4.5,2) arc [radius=2, start angle=90, end angle=270];

\draw[very thick] (-4.5,0) circle [radius=.4];
\draw[very thick] (-2,0) circle [radius=.4];
\draw[very thick] (2,0) circle [radius=.4];

\draw[thick, dotted] (-0.5,0) -- (0.5,0);

\draw[<-,thick] (1.2,0) to [out=90,in=180] (2,0.8);
\draw[thick] (2.8,0) to [out=90,in=0] (2,0.8);
\draw[thick] (1.2,0) to [out=-90,in=180] (2,-0.8) to [out=0,in=-90] (2.8,0);

\draw[<-, thick] (-5.3,0) to [out=90,in=180] (-4.5,0.8);
\draw[thick] (-3.7,0) to [out=90,in=0] (-4.5,0.8);
\draw[thick] (-5.3,0) to [out=-90,in=180] (-4.5,-0.8) to [out=0,in=-90] (-3.7,0);

\draw[<-,thick] (-2.8,0) to [out=90,in=180] (-2,0.8);
\draw[thick] (-1.2,0) to [out=90,in=0] (-2,0.8);
\draw[thick] (-2.8,0) to [out=-90,in=180] (-2,-0.8) to [out=0,in=-90] (-1.2,0);

\draw[thick] (-4.5,-0.4) to [out=180,in=90] (-5,-1.2);
\draw[->,thick] (-4.5,-2) to [out=180,in=-90] (-5,-1.2);
\draw[thick, dashed] (-4.5,-0.4) to [out=0,in=0] (-4.5,-2);
\draw[thick] (-2,-0.4) to [out=180,in=90] (-2.5,-1.2);
\draw[->,thick] (-2,-2) to [out=180,in=-90] (-2.5,-1.2);
\draw[thick, dashed] (-2,-0.4) to [out=0,in=0] (-2,-2);
\draw[thick] (2,-0.4) to [out=180,in=90] (1.5,-1.2);
\draw[->,thick] (2,-2) to [out=180,in=-90] (1.5,-1.2);
\draw[thick, dashed] (2,-0.4) to [out=0,in=0] (2,-2);

\node [left] at (-5,-1.2) {$b_1$};
\node [above] at (-4.5,0.8) {$a_1$};
\node [left] at (-2.5,-1.2) {$b_2$};
\node [left] at (1.5,-1.2) {$b_g$};
\node [above] at (-2,0.8) {$a_2$};
\node [above] at (2,0.8) {$a_g$};

\draw[thick,pattern=north west lines] (5,-2) to [out=130,in=-130] (5,2) to [out=-50,in=50] (5,-2);

\end{tikzpicture}
\end{center}
\caption{Symplectic basis of $H_1(\Sigma_{g,1};\mathbb{Z})$}
\label{fig:homology_basis}
\end{figure}
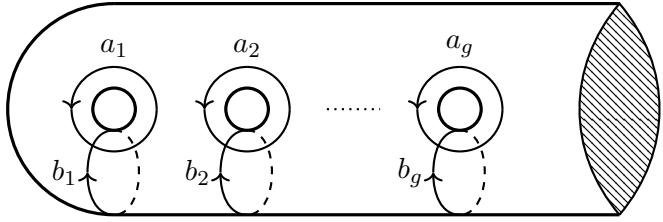

The action of the mapping class group $\mathcal{M}_{g,1}$ on $H_1(\Sigma_{g,1};\mathbb{Z})$ gives the Symplectic representation of $\mathcal{M}_{g,1}$, that is, a surjective map $\mathcal{M}_{g,1}\rightarrow Sp_g(\mathbb{Z})$ , where $Sp_g(\mathbb{Z})$ denotes the Symplectic group of matrices of size $g\times g$.
In all this work we write elements in $Sp_g(\mathbb{Z})$ by block matrices with respect the decomposition $H_1(\Sigma_{g,1};\mathbb{Z})=A \oplus B$.
In particular, for the action of elements from $\mathcal{AB}_{g,1}$,
since such elements keep invariant the Lagrangians $A$ and $B$, their image by the symplectic representation takes the form:
\[
\left(
\begin{matrix}
G & 0 \\
0 & {}^tG^{-1}
\end{matrix}
\right)
\quad \text{with} \quad G\in GL_g(\mathbb{Z}),
\]
where $GL_g(\mathbb{Z})$ denotes the General linear group of matrices $g\times g$.
By convention  when specifying an element $G\in GL_g(\mathbb{Z})$ by its action on the basis of $H_1(\Sigma_{g,1};\mathbb{Z})$, we will only write the action on the elements that are not sent to themselves.

\subsection{The Tree Lie algebra}

In this work we will perform the majority of computations in the Tree Lie algebra. For this reason we dedicate this section to exhibit the definition of such Lie algebra as well as some of its properties. For a more detailed review we refer the reader to \cite[Sec. 3]{pitsch3}.
 
Let $H=H_1(\Sigma_{g,1},\mathbb{Z})$, the Tree Lie algebra
$\mathcal{A}(H)$ is the Lie algebra of uni-trivalent trees with cyclically oriented inner vertices that are labelled by elements of $H$ modulo linearity in the labels and  the so-called IHX and AS relations:

\[
IHX:\quad
\begin{tikzpicture}[baseline=-3pt, scale=.6]
\draw[thick] (-0.5,0.5) -- (-0.5,-0.5);
\draw[thick] (-1,0.5) -- (0, 0.5);
\draw[thick] (-1,-0.5) -- (0, -0.5);
\end{tikzpicture}\;
=
\;
\begin{tikzpicture}[baseline=-3pt, scale=.6]
\draw[thick] (-1,0) -- (0,0);
\draw[thick] (-1,0.5) -- (-1, -0.5);
\draw[thick] (0,0.5) -- (0, -0.5);
\end{tikzpicture}
\;
-
\;
\begin{tikzpicture}[baseline=-3pt, scale=.6]
\draw[thick] (-0.8,0.3) -- (-0.2,0.3);
\draw[thick] (-1,0.5) -- (0, -0.5);
\draw[thick] (0,0.5) -- (-0.4, 0.1);
\draw[thick] (-0.6, -0.1) -- (-1,-0.5);
\end{tikzpicture}
\qquad \qquad AS:\quad
\begin{tikzpicture}[baseline=-3pt, scale=.6]
    \draw[thick] (0,0) -- (90:0.6cm);
    \draw[thick] (0,0) -- (-50:0.6cm);
    \draw[thick] (0,0) -- (230:0.6cm);

    \node[below] at (210:0.6cm) {$T$};
    \node[below] at (-30:0.6cm) {$T$};
  \end{tikzpicture}=0.
\]

The Lie bracket on $\mathcal{A}(H)$ is defined as follows:
given two trees $T_1,T_2\in \mathcal{A}(H)$, denote by $V_1(T_1)$ and
$V_1(T_2)$ the sets of univalent vertices of the trees $T_1$ and $T_2$ respectively. For fixed univalent vertices $x\in V_1(T_1)$ and $y\in V_1(T_2)$ denote by
$T_1-xy-T_2$ the tree that results of forgetting the label $l_x$ of $x$ and $l_y$ of $y$ and gluing the univalent vertices $x$ and $y$.
Extending this construction linearly for all univalent vertices of $T_1$ and $T_2$ we get the Lie bracket of $\mathcal{A}(H)$,
\[
[T_1,T_2]=\sum_{\substack{x\in V_1(T_1) \\ y\in V_1(T_2)}}
\omega(l_x,l_y)\;T_1-xy-T_2.
\]

Remember that, since $\mathcal{A}(H)$ is a Lie algebra, we have the \textit{Jacobi identity}:
\[
[[X,Y],Z]+[[Z,X],Y]+[[Y,Z],X]=0 \quad \text{for} \; X,Y,Z\in \mathcal{A}(H).
\]
The Jacobi identity on $\mathcal{A}(H)$ is a consequence of the IHX relation.

The Tree Lie algebra $\mathcal{A}(H)$ is a graded Lie algebra where the grading is given by the number of trivalent vertices of the trees. We denote by $\mathcal{A}_k(H)$ the degree $k$ of this Lie algebra, that is, the subset formed by trees with $k$ trivalent vertices and $k+2$ univalent vertices. From now on we call \textit{leaves} the univalent vertices of a tree, and \textit{elementary trees} those trees that in each leaf have a generator of the first homology group $H_1(\Sigma_{g,1},\mathbb{Z})$.

In this work we will perform computations with trees of degree up to $3$.
In order to help the reader to follow our computations we exhibit here the AS and IHX relations in $\mathcal{A}_k(H)$ for $k\leq 3$ as well as some of their consequences.
Let $a,b,c,d,e\in H$, for degrees $1$, $2$ and $3$, we have the following relations:
\vspace{0.3cm}

 \textit{AS relation and consequences}:
\[
\begin{array}{c}
\tree{a,b,a}=0, \qquad \tree{a,b,c}=-\tree{c,b,a}, \\
\tree{a,a,c,d}=0, \qquad \tree{a,b,c,d}=-\tree{b,a,c,d},\qquad
\tree{a,b,c,d}=\tree{d,c,b,a}, \\
\tree{a,a,c,d,e}=0, \qquad \tree{a,b,c,d,e}=-\tree{b,a,c,d,e}\quad \text{and} \quad
\tree{a,b,c,d}=-\tree{e,d,c,b,a}.
\end{array}
\]

\textit{IHX relation and consequences}:
\[
\tree{a,b,c,d}=\tree{a,c,b,d}+\tree{c,b,a,d},
\]
\[
\tree{a,b,c,d,e}=\tree{a,c,b,d,e}+\tree{c,b,a,d,e}
\quad \text{and} \quad
\tree{a,b,c,d,e}=\tree{a,b,d,c,e}+\tree{a,b,e,d,c}.
\]

\subsection{The Tree level Johnson homomorphism}
\label{subsec-tree}
In this section we review the definition of the Johnson homomorphisms and its relation with the Tree Lie algebra. For a more detailed exposition we refer the reader to \cite{pitsch3}.
Let $\mathcal{L}_k(H)$ denote the quotient $\Gamma_k\pi_1(\Sigma_{g,1})/\Gamma_{k+1}\pi_1(\Sigma_{g,1})$, where $\Gamma_k(\pi_1(\Sigma_{g,1}))$
denotes the $k$-th commutator of $\pi_1(\Sigma_{g,1})$. The group of derivations
$D_k(H)$ is the kernel of the map $[\;,\;]:H\otimes \mathcal{L}_{k+1}(H)\rightarrow \mathcal{L}_{k+2}(H)$.
In \cite[Lem. 2C]{jon_1} Johnson defined the first Johnson homomorphism $\tau_1:\mathcal{T}_{g,1}\rightarrow Hom (H,\Lambda^2H)$.
By tensor-hom adjunction this last group is isomorphic to $H\otimes\Lambda^2H$. Then in
\cite[Thm. 1]{jon_1} he proved that the image of $\tau_1$ is $\Lambda^3H$,
where $\Lambda^3H$ is embedded in $H\otimes \Lambda^2H$ by sending $a\wedge b\wedge c$ to the element
$a\otimes[b,c]+b\otimes [c,a]+c\otimes [a,b]$ and hence $\Lambda^3H$ is identified with $D_1(H)$.

Later, in \cite[Sec. 2]{mor-ab} Morita generalized the Johnson homomorphism as a series of homomorphisms $\tau_k: \mathcal{M}_{g,1}(k)\rightarrow Hom(H,\mathcal{L}_{k+1}(H))$, called the higher Johnson homomorphims. Then in \cite[Cor. 3.2]{mor-ab} he proved that the image of these homomorphisms are respectively in the group of derivations $D_k(H)$ and hence the $k$-th Johnson homomorphism takes the form: $\tau_k:\mathcal{M}_{g,1}(k)\rightarrow D_k(H)$.
Moreover, by \cite[Thm. 4.8]{mor-ab} these homomorphisms assemble into a homomorphism of graded Lie algebras
\[
\tau: \oplus_k\mathcal{M}_{g,1}(k)/\mathcal{M}_{g,1}(k+1)\rightarrow \oplus_k D_k(H).
\]
The group of derivations $D_1(H)$ and $D_2(H)$ have explicit descriptions respectively given by Johnson \cite[Sec. 4]{jon_1},
and Faes \cite[Prop. 2.1]{faes1}.

One can embed the Lie algebra $D(H)=\oplus_k D_k(H)$ into the Lie algebra of trees $\mathcal{A}(H)$.
To be more precise,
recall that for each $k\geq 1$, the tensor-hom adjunction gives an isomorphism, $Hom(H,\mathcal{L}_{k+1})\simeq H \otimes \mathcal{L}_{k+1}$. Moreover $\mathcal{L}_{k+1}$ is canonically identified with the abelian group of rooted trees $\mathcal{A}^r_{k+1}(H)$, which is formed by trees with $k+1$ trivalent vertices and a distinguished univalent vertex called the root and represented by $\ast$. In particular, we can view $D_k(H)$ as
a submodule of $H\otimes \mathcal{A}^r_{k+1}(H)$. Then if to each elementary tensor $u\otimes T\in H\otimes \mathcal{A}^r_{k+1}(H)$ we associate the tree $T_u\in \mathcal{A}_{k}(H)$, obtained by labelling the root of $T$ by $u$, extending by linearity we get the ``labelling map'',
\[
Lab: H\otimes \mathcal{A}^r_{k+1}(H) \rightarrow \mathcal{A}_k(H).
\]

Moreover, for any $k$, there are maps
\begin{align*}
\eta_k: \mathcal{A}_k(H) & \rightarrow D_k(H) \\
T & \longmapsto \sum_{x\in v(T)} l_x\otimes T^x,
\end{align*}
where $l_x$ is the element of $H$ labelling the vertex $x$ and $T^x$ is the rooted tree obtained by setting $x$ to be the root in $T$, which can be done inductively by considering that the rooted tree $(*,a,b)$ corresponds to $[b,a]$. These maps assemble into a graded Lie algebra morphism which we refer to as ``the expansion map''. The Labelling map and the expansion map are not isomorphisms but they are rational isomorphisms. To be more precise,
\begin{lemma}[Lemma 3.1 in \cite{pitsch}]
\label{lema:exp-lab}
The expansion and the labelling map have the following properties:
\begin{enumerate}[i)]
\item For any $k\geq 1$, the composite
\[ \mathcal{A}_k(H)\xrightarrow{\eta_k} D_k(H) \xrightarrow{Lab} \mathcal{A}_k(H) \]
is multiplication by $k+2$. In particular, the involved maps are rational isomorphisms.
\item The Labelling map Lab: $D_k(H) \rightarrow \mathcal{A}_k(H)$ is injective.
\end{enumerate}
\end{lemma}

In all this work we refer to the Johnson homomorphism $\tau_k$ as the composition of maps
\begin{equation}
\label{eq:def-tau-k}
 \mathcal{M}_{g,1}(k)\xrightarrow{\tau_k} D_k(H) \xrightarrow{i} D_k(H)\otimes \mathbb{Q} \xrightarrow{\eta_k^{-1}} \mathcal{A}_k(H)\otimes \mathbb{Q}.
\end{equation}

By \cite[Thm. 3.2]{pitsch3} and \cite[Thm. 4.8]{mor-ab}, the expansion map $\eta$ and the Johnson homomorphism $\tau$ are homomorphisms of Lie algebras. Then, if we denote $\Gamma_k$ the $k$-tuple Lie bracket, the map \eqref{eq:def-tau-k} restricted to $\mathcal{T}_{g,1}(k)$ becomes:
\begin{equation}
\label{eq:def-tau-k-torelli}
 \mathcal{T}_{g,1}(k)\xrightarrow{\tau_k} \Gamma_k(D_1(H))\xrightarrow{\eta_1^{-1}} \Gamma_k(\mathcal{A}_1(H))\subset \mathcal{A}_k(H).
\end{equation}

In this work we are interested in the case $k=3$. By \cite[Cor. 4.6]{FMS}, in this case we know that $\tau_3(\mathcal{M}_{g,1}(3))=\tau_3(\mathcal{T}_{g,1}(3))$ and hence by \eqref{eq:def-tau-k-torelli} the third Johnson homomorphism takes the form, $\tau_3:\mathcal{M}_{g,1}(3)\rightarrow \mathcal{A}_3(H)$.

\subsection{Coinvariants in tensors of the defining representation of $GL_g(\mathbb{Z})$}\label{sec:TensorandTrees}

In this section we compute the coinvariants module $(\otimes^{n} H)_{GL_g(\mathbb{Z})}$.
This computation is analogous to the one given in \cite[Prop. 2.4]{PR} where it was carried out over the rationals. Hence, we focus on the part of the computation that differs from the rational case and conclude as in \cite[Prop. 2.4]{PR}.

Recall that we have fixed a symplectic basis $\{a_i,b_i\}_{1 \leq i \leq g}$ of $H$. We will call an elementary tensor $u_1\otimes u_2 \dots \otimes u_p \in  \otimes^p H$ a \emph{basic tensor} if all its entries $u_i$ belong to our preferred symplectic basis. Moreover, we call such element \emph{balanced} if, for each integer $\leq i \leq g$, the number of entries $a_i$ and $b_i$ are the same.

\begin{proposition}\label{prop:chords}
For any integer $1 \leq n < g$, the coinvaiants quotient $(\otimes^{n} H)_{GL_g(\mathbb{Z})}$ is zero for $n$ odd. And for $n$ even it is generated by the images of the basic tensors  for which  for each $1 \leq  i \leq g$, either the pair of elements $\{a_i,b_i\}$ appears exactly once or does not appear at all.

Moreover, all basic tensors that are not balanced are zero in the coinvariants quotient.
\end{proposition}

\begin{proof}
The canonical action of the symmetric group $\mathfrak{S}_{k}$ by permuting the factors of $\otimes^kH$ commutes with the $GL_g(\mathbb{Z})$ action and hence induces an action on the coinvariants quotient $(\otimes^{n} H)_{GL_g(\mathbb{Z})}$. We call \emph{permutation equivalent} any two elements that are related by such a permutation. In particular, any basic tensor is permutation equivalent to a basic tensor of the form:
\[
(\otimes^{p_1}a_1)\otimes (\otimes^{q_1}b_1)\otimes \cdots \otimes (\otimes^{p_g}a_g)\otimes (\otimes^{q_g}b_g),
\]
where $p_i,q_i$ are non-negative integers with $\sum_{i=1}^{g}(p_i+q_i)=n$.


Since $n<g$, up to a permutation of the indices, we may assume that $p_g+q_g=0$. To make the notation lighter in all what follows we set $R :=\bigotimes_{i=2}^{g-1}((\otimes^{p_i}a_i)\otimes (\otimes^{q_i}b_i))$, so any basic tensor is permutation equivalent to
\[
(\otimes^{p_1}a_1)\otimes (\otimes^{q_1}b_1)\otimes R.
\]

The only point that we have to prove is that the basic tensors that are permutation equivalent to one for which $\lbrace p_1=0,\;q_1> 0\rbrace$ or $\lbrace q_1=0,\; p_1> 0\rbrace$ are sent to $0$ in  $(\otimes^{n} H)_{GL_g(\mathbb{Z})}$.
Then we conclude by Cases 3 and 4 in the proof of \cite[Prop. 2.4]{PR}.
 We prove this for $q_1=0,\; p_1>0$, so for a basic tensor permutation equivalent to $(\otimes^{p_1}a_1)\otimes R$, the other cases are analogous.
 
Consider the basic tensor
$ a_{g}\otimes(\otimes^{p_1-1}a_1)\otimes R, $
and the action by the elementary matrix $Id + E_{1,g}\in GL_g(\mathbb{Z})$. Then in $(\otimes^{n} H)_{GL_g(\mathbb{Z})}$ we have that
\[
 \begin{aligned}
 	a_{g}\otimes(\otimes^{p_1-1}a_1)\otimes R
 	& = (Id + E_{1,g}) (a_{g}\otimes(\otimes^{p_1-1}a_1)\otimes R) \\
 	& = (a_g + a_1)\otimes(\otimes^{p_1-1}a_1)\otimes R \\
 	& = a_g\otimes(\otimes^{p_1-1}a_1)\otimes R + (\otimes^{p_1}a_1)\otimes R,
 \end{aligned}
\]

and hence, in $(\otimes^{n} H)_{GL_g(\mathbb{Z})}$,
we get that
$(\otimes^{p_1}a_1) \otimes R=0.$

From here following the same lines given in Cases 3 and 4 of \cite[Prop. 2.4]{PR} we get the result for $n$ even.
Finally, for $n$ odd, we use the same arguments to show that $(\otimes^{n} H)_{GL_g(\mathbb{Z})}$ is zero:  Given an elementary tensor of $(\otimes^{n} H)_{GL_g(\mathbb{Z})}$, there exist $i\in \mathbb{N}$ for which $p_i$ and $q_i$ do not coincide since $n$ is odd. If $p_i$ or $q_i$ are zero, we have already shown that the elementary tensor is zero in the coinvariants module.
And if $p_i,q_i>0$, by Case 4 of \cite[Prop. 2.4]{PR} such element is a sum of elementary tensors with $p_i=1,q_i>1$ or $p_i>1,q_i=1$, which are zero in the coinvariants module by by Case 3 of \cite[Prop. 2.4]{PR}.

\end{proof}

\section{The image of the third Johnson homomorphism}
We start this section by introducing some notation.
Let $i,j,k\in \mathbb{N}$ with $i+j-2=k$, we denote by $W(a^ib^j)$ the submodule of $\mathcal{A}_k(H)$ generated by trees with $i$ leaves coloured with $a$'s and $j$ leaves coloured with $b$'s.
Moreover, let $r,s\in \mathbb{N}$, we denote by
$W_k(a^{\geq r} b)$ and $W_k(ab^{\geq s})$ the submodules of $\mathcal{A}_k(H)$ generated by the trees with at least $r$ leaves colored with $a$'s, respectively $s$ leaves colored with $b$'s.
We also denote by $W_k(a^{\geq r} b^{\geq s})$ the intersection $W_k(a^{\geq r} b) \cap W_k(ab^{\geq s})$. To make notations lighter when $k=1$, we simply write $W$ instead of $W_1$. 
In addition, we set $\overline{W}$ instead of $W$ to denote the intersection of these modules with $Im(\tau_k)$. For instance, $\overline{W}(a^ib^j)=W(a^ib^j)\cap Im(\tau_k)$.

We now focus in our case of interest, $k=3$.
If we gather the trees with the same colouring of leaves we get a decomposition as $GL_g(\mathbb{Z})$-modules,
\[
\mathcal{A}_3(H)=W(a^5)\oplus W(a^4b^1)\oplus W(a^3b^2)\oplus W(a^2b^3)\oplus W(a^1b^4)\oplus W(b^5).
\]

In this section we show that the aforementioned decomposition induces an analogous decomposition for the image of the third Johnson homomorphism, that is, any element of $Im(\tau_3)$ can be expressed as a sum of elements of $\overline{W}(a^ib^j)$ with $i+j=5$.
Then we reduce the number of generators of each module $\overline{W}(a^ib^j)$ by using the Jacobi identity and exchanging the colours of some leaves that are contracted by the Lie bracket. 

\begin{proposition}
\label{prop:desc-m(3)}
For any integer $g\geq 5$, there is a decomposition of $GL_g(\mathbb{Z})$-modules,
\[
\tau_3(\mathcal{M}_{g,1}(3))=\overline{W}(a^5)\oplus \overline{W}(a^4b^1)\oplus \overline{W}(a^3b^2)\oplus \overline{W}(a^2b^3)\oplus \overline{W}(a^1b^4)\oplus \overline{W}(b^5).
\]

Moreover, the modules $\overline{W}(a^hb^{5-h})$ with $0\leq h\leq 5$ are the sum of modules:
\[
[[W(a^ib^{3-i}), W(a^jb^{3-j})], W(a^kb^{3-k})]\quad \text{with }
\begin{array}{c}
i,j,k\in \lbrace 0,1,2,3 \rbrace, \\
h=i+j+k-2.
\end{array}
\]
\end{proposition}

\begin{proof}
Form the definition of $\overline{W}(a^ib^j)$ it is clear that $\overline{W}(a^ib^j)\subset \tau_3(\mathcal{M}_{g,1}(3))$.
We now prove that any element of $\tau_3(\mathcal{M}_{g,1}(3))$ can be written as a sum of elements in $\overline{W}(a^ib^j)$ with $i,j\in \mathbb{N}$ and $i+j=5$.

By \cite[Cor. 4.6]{FMS} we know that $\mathcal{M}_{g,1}(3)=\mathcal{T}_{g,1}(3)\cdot \mathcal{M}_{g,1}(4)$. Since $\mathcal{M}_{g,1}(4) =Ker(\tau_3)$, applying the third Johnson homomorphism $\tau_3$, we get that $\tau_3(\mathcal{M}_{g,1}(3))= \tau_3(\mathcal{T}_{g,1}(3))$.
By definition of $\mathcal{T}_{g,1}(3)$, any element of this group is a product of triple commutators of elements of $\mathcal{T}_{g,1}$.
Since, by \cite[Thm. 4.8]{mor-ab}, the Johnson homomorphism $\tau_k$ is a homomorphism of Lie algebras, then we get the following equalities:
\begin{equation}
\label{eq:com-3-john}
\tau_3(\mathcal{T}_{g,1}(3))=\tau_3([[\mathcal{T}_{g,1},\mathcal{T}_{g,1}],\mathcal{T}_{g,1}]])=[[\tau_1(\mathcal{T}_{g,1}),\tau_1(\mathcal{T}_{g,1})],\tau_1(\mathcal{T}_{g,1})].
\end{equation}

By \cite[Thm 1]{jon_1} we know that the first Johnson homomorphism $\tau_1: \mathcal{T}_{g,1}\rightarrow \mathcal{A}_1(H)$ is surjective. Moreover,
there is a decomposition as $GL_g(\mathbb{Z})$-modules:
\begin{equation}
\label{eq:dec-lambda3}
\mathcal{A}_1(H)= W(a^3)\oplus W(a^2b)\oplus W(ab^2)\oplus W(b^3).
\end{equation}

Then the group $\tau_3(\mathcal{T}_{g,1}(3))$ is the sum of the images of the triple Lie bracket 
\[ [[\;,\;],\;]: (\mathcal{A}_1(H)\wedge \mathcal{A}_1(H))\otimes \mathcal{A}_1(H)\rightarrow \mathcal{A}_3(H)\]
on the modules $(W(a^ib^{3-i}) \wedge W(a^jb^{3-j})) \otimes W(a^kb^{3-k})$ with $i,j,k\in \lbrace 0,1,2,3 \rbrace$, and these images are contained in $\overline{W}(a^hb^{5-h})$ with $h=i+j+k-2$ such that $0\leq h\leq 5$.

Moreover, if we gather the resulting trees of the triple Lie bracket with the same coloring of leaves, we have that the module $\overline{W}(a^hb^{5-h})$ with $0\leq h\leq 5$ is the sum of the modules
$[[W(a^ib^{3-i}), W(a^jb^{3-j})], W(a^kb^{3-k})]$ with $i,j,k\in \lbrace 0,1,2,3 \rbrace$ and $h=i+j+k-2$.
\end{proof}

\subsection{Lie brackets of trees}
\label{Sec:brackets-trees}

In all what follows, to make notations lighter, we define $W(a^{\geq 1}b^{\geq 1})=W(a^2b)\oplus W(ab^2)$ and we denote by 
$\Gamma_2(W(a^{\geq 1}b^{\geq 1}))$ and
$\Gamma_3(W(a^{\geq 1}b^{\geq 1}))$, respectively, the double and triple Lie brackets of elements in $W(a^{\geq 1}b^{\geq 1})$.

We reduce the number of generators of the modules $\overline{W}(a^hb^{5-h})$ with $0\leq h\leq 5$.
In particular, we reduce the number of submodules
$[[W(a^ib^{3-i}), W(a^jb^{3-j})], W(a^kb^{3-k})]$
that generate $\overline{W}(a^hb^{5-h})$.
We first give three preliminary results about the image of the Lie bracket $[\;,\;]:\mathcal{A}_1(H)\wedge\mathcal{A}_1(H)\rightarrow \mathcal{A}_2(H)$ on modules $[W(a^ib^{3-i}), W(a^jb^{3-j})]$.

\begin{lemma}
\label{lema:[A,B]}
For any integer $g\geq 5$, we have that 
\[
{[}W(a^3), W(b^3)]\subset [W(ab^2), W(a^2b)].
\]
\end{lemma}

\begin{proof}
Given an integer $g\geq 5$, the module $[W(a^3), W(b^3)]$ is generated by
the Lie brackets of the form $\Bigg[\tree{a_i,a_j,a_k},\tree{b_{i'},b_{j'},b_{k'}}\Bigg]$ with $1\leq i<j<k\leq g$ and
$1\leq i'<j'<k'\leq g$.
We show that all these brackets belong to $[W(ab^2), W(a^2b)]$. We analyse these brackets in terms of the number of contractions between the trees $\tree{a_i,a_j,a_k}$ and $\tree{b_{i'},b_{j'},b_{k'}}$.
If there are no contractions between these two trees then, by construction, their Lie bracket is zero.

\textbf{Three contractions.} If there are exactly $3$ contractions between $\tree{a_i,a_j,a_k}$ and $\tree{b_{i'},b_{j'},b_{k'}}$, we are in the $\mathfrak{S}_g$-orbit the element:
\begin{equation}
\label{eq:modWAB1}
\Bigg[\tree{a_1,a_2,a_3},\tree{b_{3},b_{2},b_{1}}\Bigg]=\tree{a_2,a_3,b_3,b_2}+\tree{a_3,a_1,b_1,b_3}+\tree{a_1,a_2,b_2,b_1}.
\end{equation}
Next we show that this element belongs to $[W(ab^2), W(a^2b)]$.
Consider the following sum of brackets:
\[
\begin{aligned}
& \Bigg[\tree{a_1,a_2,b_3},\tree{a_{3},b_{2},b_{1}}\Bigg]
-\Bigg[\tree{b_3,a_3,a_1},\tree{b_1,b_2,a_2}\Bigg]
-\Bigg[\tree{a_1,b_1,a_2},\tree{b_2,a_3,b_3}\Bigg] \\
& = \tree{a_2,b_3,a_3,b_2}+\tree{b_3,a_1,b_1,a_3}-\tree{a_1,a_2,b_2,b_1}
-\tree{b_3,a_3,b_2,a_2}-\tree{a_1,b_1,a_3,b_3}\\
& =  \tree{a_2,a_3,b_3,b_2}+\tree{a_3,a_1,b_1,b_3}-\tree{a_1,a_2,b_2,b_1},
\end{aligned}
\]
where in the last equality we used the IHX relation in the first two summands.
This shows,
\begin{equation}
\label{eq:modWAB2}
 \tree{a_2,a_3,b_3,b_2}+\tree{a_3,a_1,b_1,b_3}-\tree{a_1,a_2,b_2,b_1}\in
 [W(ab^2), W(a^2b)],
\end{equation}
and analogously, exchanging the subindices $1$ and $3$,
\begin{equation}
\label{eq:modWAB3}
\tree{a_2,a_1,b_1,b_2}+\tree{a_1,a_3,b_3,b_1}-\tree{a_3,a_2,b_2,b_3}\in
 [W(ab^2), W(a^2b)].
\end{equation}

Then, by the AS relation, the sum of \eqref{eq:modWAB2} and \eqref{eq:modWAB3}, gives us that
$ 2 \tree{a_1,a_3,b_3,b_1}$ belongs to 
$[W(ab^2), W(a^2b)]$, and exchanging $3$ and $2$ we obtain that
\begin{equation}
\label{eq:modWAB4}
2 \tree{a_1,a_2,b_2,b_1}\in [W(ab^2), W(a^2b)].
\end{equation}

Finally, the sum of \eqref{eq:modWAB2} and \eqref{eq:modWAB4} is the element \eqref{eq:modWAB1}. Therefore,
\[
\Bigg[\tree{a_1,a_2,a_3},\tree{b_{1},b_{2},b_{3}}\Bigg]\in [W(ab^2), W(a^2b)].
\]

\textbf{Two contractions.} If there are exactly $2$ contractions between $\tree{a_i,a_j,a_k}$ and $\tree{b_{i'},b_{j'},b_{k'}}$, we are in the $\mathfrak{S}_g$-orbit the element:
\[
\Bigg[\tree{a_1,a_2,a_3},\tree{b_{1},b_{2},b_{4}}\Bigg].
\]
Next we show that this element belongs to $[W(ab^2), W(a^2b)]$.
Consider the element $\Bigg[\tree{a_1,a_2,a_3},\tree{b_{1},b_{2},b_{3}}\Bigg]$ which is in $[W(ab^2), W(a^2b)]$ by the above computations. Taking the action by $G\in GL_g(\mathbb{Z})$ that sends $b_3$ to $b_3+b_4$ and $a_4$ to $a_4-a_3$ we have that
\[
G\Bigg(\Bigg[\tree{a_1,a_2,a_3},\tree{b_{1},b_{2},b_{3}}\Bigg]\Bigg)= \Bigg[\tree{a_1,a_2,a_3},\tree{b_{1},b_{2},b_{3}}\Bigg]
+\Bigg[\tree{a_1,a_2,a_3},\tree{b_{1},b_{2},b_{4}}\Bigg].
\]

Since the action by $GL_g(\mathbb{Z})$ preserves the module $[W(ab^2), W(a^2b)]$, we have that the element $\Bigg[\tree{a_1,a_2,a_3},\tree{b_{1},b_{2},b_{4}}\Bigg]$ belongs to $[W(ab^2), W(a^2b)]$.

\textbf{One contraction.} If there is a single contraction between $\tree{a_i,a_j,a_k}$ and $\tree{b_{i'},b_{j'},b_{k'}}$, we are in the $\mathfrak{S}_g$-orbit of the element:
\[
\Bigg[\tree{a_1,a_2,a_3},\tree{b_{1},b_{4},b_{5}}\Bigg]=-\Bigg[\tree{b_1,a_2,a_3},\tree{a_{1},b_{4},b_{5}}\Bigg]\in [W(ab^2), W(a^2b)].
\]
\end{proof}

\begin{lemma}
\label{lema:[A,ABB]}
For any integer $g\geq 3$, the module $[W(a^3), W(ab^2)]$ modulo $[W(a^2b), W(a^2b)]$ and the module $[W(b^3), W(a^2b)]$ modulo $[W(ab^2), W(ab^2)]$ are respectively generated by:
\[
2\tree{a_i,a_1,b_{1},a_j}\quad\text{and}\quad 2\tree{b_i,a_1,b_{1},b_j}.
\]
\end{lemma}

\begin{proof}
We prove the statement for the module $[W(a^3), W(ab^2)]$. For the other module the proof is analogous.
Given an integer $g\geq 3$, the module $[W(a^3), W(ab^2)]$ is generated by
the brackets of the form $\Bigg[\tree{a_i,a_j,a_k},\tree{a_{i'},b_{j'},b_{k'}}\Bigg]$ with $1\leq i,i',j,j',k,k'\leq g$.
We show that all these brackets belong to $[W(ab^2), W(a^2b)]$.
We analyse these brackets in terms of the number of contractions between the tree $\tree{a_i,a_j,a_k}$ and $\tree{a_{i'},b_{j'},b_{k'}}$.
If there are no contractions between these two trees, their bracket is zero.

\textbf{One contraction.} If there is a single contraction between $\tree{a_i,a_j,a_k}$ and $\tree{a_{i'},b_{j'},b_{k'}}$, we are in the $\mathfrak{S}_g$-orbit of the element:
\[
\Bigg[\tree{a_1,a_2,a_3},\tree{a_i,b_{1},b_{j}}\Bigg]=-\Bigg[\tree{b_1,a_2,a_3},\tree{a_i,a_{1},b_{j}}\Bigg]\in [W(a^2b), W(a^2b)],
\]
with $1\leq i,j,\leq g$ such that $i\neq 1$ and $j \neq 1,2,3$.

\textbf{Two contractions.} If there are exactly $2$ contractions between $\tree{a_i,a_j,a_k}$ and $\tree{a_{i'},b_{j'},b_{k'}}$, we are in the $\mathfrak{S}_g$-orbit of the element:
\begin{equation}
\label{eq:2-cont-A-ABB}
\Bigg[\tree{a_1,a_2,a_3},\tree{a_{i},b_{2},b_{1}}\Bigg]= \tree{a_3,a_1,b_1,a_i}+\tree{a_3,a_2,b_2,a_i}, \quad \text{with }1\leq i\leq g.
\end{equation}

Consider the element of $[W(a^2b), W(a^2b)]$,
\begin{equation}
\label{eq:eq:2-cont-A-ABB-2}
\begin{aligned}
\Bigg[\tree{b_1,a_1,a_2},\tree{b_2,a_i,a_3}\Bigg]-\Bigg[\tree{b_1,a_2,a_3},\tree{a_{i},b_{2},a_{1}}\Bigg]
= & \;\tree{b_1,a_1,a_i,a_3}-\tree{a_3,b_1,a_1,a_i}+\tree{a_2,a_3,a_i,b_2} \\
= & \; -\tree{a_3,a_1,b_1,a_i}+\tree{a_3,a_2,b_2,a_i},
\end{aligned}
\end{equation}
where in the second equality we used the IHX relation in first tree and the AS relation in the third tree.
Adding the element \eqref{eq:eq:2-cont-A-ABB-2} to \eqref{eq:2-cont-A-ABB} we get that, modulo $[W(a^2b), W(a^2b)]$,
\[
\Bigg[\tree{a_1,a_2,a_3},\tree{a_{i},b_{2},b_{1}}\Bigg]= 2\tree{a_3,a_1,b_1,a_i}.
\]
\end{proof}

\begin{lemma}
\label{lema:AABB-elements}
For any integer $g\geq 4$, the following trees belong to $[W(a^2b), W(ab^2)]$.
\[
2\tree{b_2,a_1,b_{1},a_4} \quad\text{and}\quad 2\tree{a_4,a_1,b_{1},b_2}.
\]
\end{lemma}

\begin{proof}
We first prove that the first tree of the statement belongs to $[W(a^2b), W(ab^2)]$. Consider the element of $[W(a^2b), W(ab^2)]$,
\[
\Bigg[\tree{b_2,a_1,a_3},\tree{b_3,b_1,a_4}\Bigg]+\Bigg[\tree{b_2,a_1,a_4},\tree{b_4,b_1,a_4}\Bigg] -\Bigg[\tree{b_2,a_3,a_4},\tree{b_4,b_3,a_4}\Bigg].
\]
This element is equal to:
\[
\tree{b_2,a_1,b_{1},a_4}+\tree{b_2,a_3,b_{3},a_4}
+\tree{b_2,a_1,b_{1},a_4}+\tree{b_2,a_4,b_{4},a_4}
-\tree{b_2,a_3,b_{3},a_4}-\tree{b_2,a_4,b_{4},a_4}
=2\tree{b_2,a_1,b_{1},a_4}.
\]
Therefore,
\[
2\tree{b_2,a_1,b_{1},a_4}\in [W(a^2b), W(ab^2)].
\]

To prove that the second tree of the statement also belongs to $[W(a^2b), W(ab^2)]$ we apply the IHX relation on the first tree of the statement, which
we already know that belongs to $[W(a^2b), W(ab^2)]$, to get that,
\[
2\tree{b_2,a_1,b_{1},a_4}= 2\tree{b_2,b_1,a_{1},a_4}+2\tree{b_1,a_1,b_{2},a_4}
= 2\tree{a_4,a_1,b_{1},b_2}+2\Bigg[\tree{b_1,a_1,a_4},\tree{b_4,b_2,a_4}\Bigg],
\]
where in the last equality we used the AS relation in the first tree.
Therefore,
\[
2\tree{a_4,a_1,b_{1},b_2}\in [W(a^2b), W(ab^2)].
\]

\end{proof}

We now focus on the image of the triple Lie bracket map
\[
[[\;,\;],\;]: (\mathcal{A}_1(H)\wedge \mathcal{A}_1(H)) \otimes \mathcal{A}_1(H)\rightarrow \mathcal{A}_3(H),
\]
to reduce the number of generators of the submodules $\overline{W}(a^hb^{5-h})$ for $3\leq h\leq 5$. The computations for the other submodules are analogous exchanging $a$'s and $b$'s.
Remember that, by Proposition \ref{prop:desc-m(3)},
the submodule $\overline{W}(a^hb^{5-h})$ is generated by 
the image of the triple Lie bracket on the modules
\begin{equation}
\label{eq:modules-color-h}
(W(a^ib^{3-i}) \wedge W(a^jb^{3-j})) \otimes W(a^kb^{3-k})
\quad \text{with }
\begin{array}{c}
i,j,k\in \lbrace 0,1,2,3 \rbrace, \\
h=i+j+k-2.
\end{array}
\end{equation}

Keep in mind that the image of these modules by the triple Lie bracket
are trivially zero unless $W(a^ib^{3-i})$ has $1$ contraction with $W(a^jb^{3-j})$ and $W(a^kb^{3-k})$ has $1$ contractions with $W(a^ib^{3-i}) \wedge W(a^jb^{3-j})$.
Moreover, the action by $GL_g(\mathbb{Z})$ on $H_1(\Sigma_{g,1},\mathbb{Z})$ induces an action on each module $(W(a^ib^{3-i}) \wedge W(a^jb^{3-j})) \otimes W(a^kb^{3-k})$.
\vspace{0.3cm}

\textbf{The submodule $\overline{W}(a^3b^2)$.}
If we gather all modules of the form \eqref{eq:modules-color-h} such that their image by the triple Lie bracket is contained in $\overline{W}(a^3b^2)$, we get that this module is generated by the image of the triple Lie bracket $[[\;,\;],\;]$ on the following modules:

\[
\begin{array}{ll}
 (W(a^3)\wedge W(a^2b))\otimes W(b^3), & (W(a^2b)\wedge W(a^2b))\otimes W(ab^2), \\[1ex]
 (W(a^3)\wedge W(ab^2))\otimes W(ab^2), & (W(a^2b)\wedge W(ab^2))\otimes W(a^2b), \\[1ex]
 (W(a^3)\wedge W(b^3)) \otimes W(a^2b), & (W(a^2b)\wedge W(b^3))\otimes W(a^3), \\[1ex]
 \multicolumn{2}{c}{(W(ab^2)\wedge W(ab^2))\otimes W(a^3).}
\end{array}
\]

We can in fact reduce the number of modules to consider.
By the Jacobi identity, the image of the module $(W(a^2b)\wedge W(b^3))\otimes W(a^3)$ by the triple Lie bracket is contained in the sum of the images of the modules $(W(a^3)\wedge W(a^2b))\otimes W(b^3)$ and $(W(a^3)\wedge W(b^3)) \otimes W(a^2b)$.
By Lemma \ref{lema:[A,B]} and the Jacobi identity, the image of the modules $(W(a^3)\wedge W(b^3)) \otimes W(a^2b)$ and $(W(a^2b)\wedge W(a^2b))\otimes W(ab^2)$ are contained in the image of $(W(a^2b)\wedge W(ab^2))\otimes W(a^2b)$. Again, by the Jacobi identity, the image of the module $(W(ab^2)\wedge W(ab^2))\otimes W(a^3)$ is contained in the image of the module $(W(a^3)\wedge W(ab^2))\otimes W(ab^2)$. Therefore the module $\overline{W}(a^3b^2)$ is generated by the image of the triple Lie bracket on the following modules:
\begin{equation}
\label{eq:list-mod-a3b2}
\begin{array}{ll}
 (W(a^3)\wedge W(a^2b))\otimes W(b^3), & (W(a^3)\wedge W(ab^2))\otimes W(ab^2), \\[1ex]
  \multicolumn{2}{c}{(W(a^2b)\wedge W(ab^2))\otimes W(a^2b).}
\end{array}
\end{equation}

As a preliminary result we show:

\begin{lemma}
\label{lema:trees-a4b2-needed}
For any integer $g\geq 5$ the following trees belong to $\Gamma_3(W(a^{\geq 1}b^{\geq 1}))$.
\[
\begin{array}{l}
\textit{i)} \tree{a_3,a_2,b_1,a_1,b_i}
\qquad \textit{ii)} \tree{a_3,a_2,a_1,b_1,b_i}
\qquad \text{for } i\neq 2,3, \\
\textit{iii)} \tree{a_2,a_3,a_1,b_1,b_3}-\tree{a_2,a_4,a_1,b_1,b_4}.
\end{array}
\]
\end{lemma}

\begin{proof}
\textbf{i)} Let $i,j,k\in \mathbb{N}$ such that $i,j,k\notin \lbrace 2,3\rbrace$ with $j\neq k$. Consider the element of $\Gamma_3(W(a^{\geq 1}b^{\geq 1}))$,
\[
\Bigg[\Bigg[\tree{a_2,b_j,a_j},\tree{b_j,a_3,b_k}\Bigg],\tree{a_k,a_j,b_i}\Bigg]=\Bigg[\tree{a_2,b_j,a_3,b_k},\tree{a_k,a_j,b_i}\Bigg]= -\tree{a_2,b_j,a_3,a_j,b_i}+\tree{a_3,b_k,a_2,a_k,b_i}.
\]

Then the following sum of elements of $\Gamma_3(W(a^{\geq 1}b^{\geq 1}))$,
\[
\Bigg[\tree{a_2,b_1,a_3,b_4},\tree{a_4,a_1,b_i}\Bigg]+\Bigg[\tree{a_3,b_4,a_2,b_5},\tree{a_5,a_4,b_i}\Bigg]
+ \Bigg[\tree{a_2,b_5,a_3,b_1},\tree{a_1,a_5,b_i}\Bigg],
\]
is equal to the sum:
\[
-\tree{a_2,b_1,a_3,a_1,b_i}+\tree{a_3,b_4,a_2,a_4,b_i}
-\tree{a_3,b_4,a_2,a_4,b_i}+\tree{a_2,b_5,a_3,a_5,b_i}
-\tree{a_2,b_5,a_3,a_5,b_i}+\tree{a_3,b_1,a_2,a_1,b_i},
\]
which in turn is equal to
\begin{equation}
-\tree{a_2,b_1,a_3,a_1,b_i}+\tree{a_3,b_1,a_2,a_1,b_i}.
\end{equation}
By the IHX relation we have that this last element is equal to:
\begin{equation*}
\tree{a_3,a_2,b_1,a_1,b_i},
\end{equation*}
which indeed belongs to $\Gamma_3(W(a^{\geq 1}b^{\geq 1}))$ for $i\neq 2,3$.
\vspace{0.5cm}

\textbf{ii)}  For $i\neq 2,3$, by the IHX relation, we have the following equality:
\[
\tree{a_3,a_2,a_1,b_1,b_i}= \tree{a_3,a_2,b_1,a_1,b_i}+ \tree{a_3,a_2,b_i,b_1,a_1},
\]
where
\[
\tree{a_3,a_2,b_i,b_1,a_1} =\Bigg[\tree{a_3,a_2,b_i,a_l},\tree{b_l,b_1,a_1}\Bigg]=-\Bigg[\Bigg[\tree{a_3,a_2,b_1},\tree{a_1,b_i,a_l}\Bigg],\tree{b_l,b_1,a_1}\Bigg]\quad \text{with } l\neq 1,2,3. 
\]
Therefore, by point $i)$ we get that,
\[
\tree{a_3,a_2,a_1,b_1,b_i}\in \Gamma_3(W(a^{\geq 1}b^{\geq 1})),
\qquad \text{for } i\neq 2,3.
\]
\vspace{0.5cm}

\textbf{iii)} Consider the element $\tree{a_2,a_4,a_1,b_1,b_3}$, which belongs to the subgroup $\Gamma_3(W(a^{\geq 1}b^{\geq 1}))$ by point $ii)$.
Act by $G\in GL_g(\mathbb{Z})$, which sends $a_4$ to $a_4+a_3$ and $b_3$ to $b_3-b_4$.
We compute:
\[
G\Bigg(\tree{a_2,a_4,a_1,b_1,b_3}\Bigg)= \tree{a_2,a_4,a_1,b_1,b_3}+ \tree{a_2,a_3,a_1,b_1,b_3} -\tree{a_2,a_4,a_1,b_1,b_4}- \tree{a_2,a_3,a_1,b_1,b_4}.
\]
Then, by point $ii)$, we get that 
\[\tree{a_2,a_3,a_1,b_1,b_3}-\tree{a_2,a_4,a_1,b_1,b_4}\in \Gamma_3(W(a^{\geq 1}b^{\geq 1})).\]
\end{proof}

We now show that the image of the modules $(W(a^3)\wedge W(a^2b))\otimes W(b^3)$ and $(W(a^3)\wedge W(ab^2))\otimes W(ab^2)$ by the triple Lie bracket are contained in the image of the module $(W(a^2b)\wedge W(ab^2))\otimes W(a^2b)$. As a consequence, we get that the submodule $\overline{W}(a^3b^2)$ is contained in $\Gamma_3(W(a^{\geq 1}b^{\geq 1}))$.

\begin{lemma}
\label{lema:[[A,AB],B]}
For any integer $g\geq 5$ we have the following inclusion of modules:
\[
[[W(a^3), W(a^2b)], W(b^3)]\subset \Gamma_3(W(a^{\geq 1}b^{\geq 1})).
\]
\end{lemma}

\begin{proof}
By construction, $[W(a^3), W(a^2b)]\subset W(a^4)$. Moreover, any tree of $W(a^4)$ can be written as an element of $[W(a^3), W(a^2b)]$ by taking the bracket of the two trees that result of
cutting the tree of $W(a^4)$ in the unique edge that joins the trivalent vertices and adding labels $a_i$ and $b_i$ to each respective created leaves.
Then $[W(a^3), W(a^2b)]= W(a^4)$.
Hence,
\[
[[W(a^3), W(a^2b)], W(b^3)]=[W(a^4), W(b^3)].
\]
The generators of this last module are the $\mathfrak{S}_g$-orbits of the following elements:
\[
\Bigg[\tree{a_i,a_j,a_k,a_l}, \tree{b_{1},b_{2},b_{3}}\Bigg]\quad \text{with }
1\leq i,j,k,l \leq g.
\]

In the sequel, for each number of contractions between the left and right trees of these generators we prove that they belong to $\Gamma_3(W(a^{\geq 1}b^{\geq 1}))$.
\vspace{0.5cm}

\textbf{Three contractions.} If there are exactly $3$ contractions between $\tree{a_i,a_j,a_k,a_l}$ and $\tree{b_{1},b_{2},b_{3}}$, modulo the AS relation, we are in the $\mathfrak{S}_g$-orbits of the elements:
\begin{equation}
\label{eq:lemma-3contract-list}
\Bigg[\tree{a_1,a_2,a_3,a_4},\tree{b_{1},b_{2},b_{3}}\Bigg]
\quad \text{and} \quad \Bigg[\tree{a_1,a_2,a_1,a_4},\tree{b_{1},b_{2},b_{3}}\Bigg].
\end{equation}

We show that the first element of \eqref{eq:lemma-3contract-list} belongs to $\Gamma_3(W(a^{\geq 1}b^{\geq 1}))$. We compute:
\[
\begin{aligned}
 \Bigg[\tree{a_1,a_2,a_3,a_4},\tree{b_{1},b_{2},b_{3}}\Bigg] = & \; \tree{a_4,a_3,a_2,b_2,b_3}+\tree{a_3,a_4,a_1,b_3,b_1} +\tree{a_2,a_1,a_4,b_1,b_2} \\
= & \tree{a_4,a_3,a_2,b_2,b_3}+\tree{a_4,a_3,a_1,b_1,b_3} +\tree{a_4,a_1,a_2,b_1,b_2}+\tree{a_2,a_4,a_1,b_1,b_2} \\ 
= & \; \tree{a_4,a_3,a_2,b_2,b_3}+\tree{a_4,a_3,a_1,b_1,b_3} -\tree{a_4,a_1,a_2,b_2,b_1}-\tree{a_4,a_2,a_1,b_1,b_2},
\end{aligned}
\]
where in the second equality we used the AS relation on the second tree and the IHX relation on the third tree, and in the last equality we used the AS relation on the third and fourth trees.
Then, by Point $iii)$ in Lemma \ref{lema:trees-a4b2-needed}, we have that this element indeed belongs to $\Gamma_3(W(a^{\geq 1}b^{\geq 1}))$.

We now prove that the second element of \eqref{eq:lemma-3contract-list} belongs to $\Gamma_3(W(a^{\geq 1}b^{\geq 1}))$. Consider the first element of \eqref{eq:lemma-3contract-list}, which we already know that belongs to $\Gamma_3(W(a^{\geq 1}b^{\geq 1}))$. Act by $G\in GL_g(\mathbb{Z})$, which sends $a_3$ to $a_3+a_1$ and $b_1$ to $b_1-b_3$. Then we compute:
\[
G\Bigg(\Bigg[\tree{a_1,a_2,a_3,a_4},\tree{b_{1},b_{2},b_{3}}\Bigg]\Bigg) =
\Bigg[\tree{a_1,a_2,a_3,a_4},\tree{b_{1},b_{2},b_{3}}\Bigg]
+\Bigg[\tree{a_1,a_2,a_1,a_4},\tree{b_{1},b_{2},b_{3}}\Bigg],
\]
where we used that, by linearity on the leaves and the AS relation, the action of the element $G$ leaves the second tree of the first bracket invariant.
This shows that the second element of \eqref{eq:lemma-3contract-list} also belongs to $\Gamma_3(W(a^{\geq 1}b^{\geq 1}))$.

\textbf{Two contractions.} If there are exactly $2$ contractions between $\tree{a_i,a_j,a_k,a_l}$ and $\tree{b_{1},b_{2},b_{3}}$, modulo the AS relation, we are in the $\mathfrak{S}_g$-orbits of the elements:
\begin{equation}
\label{eq:lemma-2contract-list}
\Bigg[\tree{a_1,a_2,a_i,a_4},\tree{b_{1},b_{2},b_{3}}\Bigg],\quad
\Bigg[\tree{a_1,a_i,a_3,a_4},\tree{b_{1},b_{2},b_{3}}\Bigg]\quad \text{and} \quad
\Bigg[\tree{a_1,a_i,a_1,a_4},\tree{b_{1},b_{2},b_{3}}\Bigg],
\end{equation}
with $i\notin \lbrace 1,2,3 \rbrace$.

We show that the first two elements of \eqref{eq:lemma-2contract-list} belong to $\Gamma_3(W(a^{\geq 1}b^{\geq 1}))$. Consider the first element of \eqref{eq:lemma-3contract-list}, which belongs to $\Gamma_3(W(a^{\geq 1}b^{\geq 1}))$ by the three contractions case.
For $i\notin \lbrace 1,2,3 \rbrace$, act by $G_1\in GL_g(\mathbb{Z})$, which sends $a_3$ to $a_3+a_i$ and $b_i$ to $b_i-b_3$ and act by $G_2\in GL_g(\mathbb{Z})$, which sends $a_2$ to $a_2+a_i$ and $b_i$ to $b_i-b_2$.
Then we compute:
\[
\begin{aligned}
 G_1\Bigg(\Bigg[\tree{a_1,a_2,a_3,a_4},\tree{b_{1},b_{2},b_{3}}\Bigg]\Bigg) = & \Bigg[\tree{a_1,a_2,a_3,a_4},\tree{b_{1},b_{2},b_{3}}\Bigg] + \Bigg[\tree{a_1,a_2,a_i,a_4},\tree{b_{1},b_{2},b_{3}}\Bigg],\\
G_2\Bigg(\Bigg[\tree{a_1,a_2,a_3,a_4},\tree{b_{1},b_{2},b_{3}}\Bigg]\Bigg)  = & \Bigg[\tree{a_1,a_2,a_3,a_4},\tree{b_{1},b_{2},b_{3}}\Bigg] 
+ \Bigg[\tree{a_1,a_i,a_3,a_4},\tree{b_{1},b_{2},b_{3}}\Bigg],
\end{aligned}
\]
where we used that, by linearity on the leaves and the AS relation, the action of the element $G_1$ and $G_2$ leave the second tree of the first bracket invariant.
This shows that the first two elements of \eqref{eq:lemma-2contract-list} belong to $\Gamma_3(W(a^{\geq 1}b^{\geq 1}))$.

We now prove that the third element of \eqref{eq:lemma-2contract-list} belongs to $\Gamma_3(W(a^{\geq 1}b^{\geq 1}))$. Consider the second element of \eqref{eq:lemma-2contract-list}, which we already know that belongs to $\Gamma_3(W(a^{\geq 1}b^{\geq 1}))$. Act by $G\in GL_g(\mathbb{Z})$, which sends $a_3$ to $a_3+a_1$ and $b_1$ to $b_1-b_3$. Then we compute:
\[
G\Bigg(\Bigg[\tree{a_1,a_i,a_3,a_4},\tree{b_{1},b_{2},b_{3}}\Bigg]\Bigg)=
\Bigg[\tree{a_1,a_i,a_3,a_4},\tree{b_{1},b_{2},b_{3}}\Bigg]+\Bigg[\tree{a_1,a_i,a_1,a_4},\tree{b_{1},b_{2},b_{3}}\Bigg],
\]
where we used that, by linearity on the leaves and the AS relation, the action of the element $G$ leaves the second tree of the first bracket invariant.
This shows that the third element of \eqref{eq:lemma-2contract-list}
also belongs to $\Gamma_3(W(a^{\geq 1}b^{\geq 1}))$.

\textbf{One contraction.} If there is a single contraction between $\tree{a_i,a_j,a_k,a_l}$ and $\tree{b_{1},b_{2},b_{3}}$, modulo the AS relation, we are in the $\mathfrak{S}_g$-orbit of the element:
\[
\Bigg[\tree{a_i,a_j,a_k,a_1},\tree{b_{1},b_{2},b_{3}}\Bigg]=-\Bigg[\tree{a_i,a_j,a_k,b_1},\tree{a_{1},b_{2},b_{3}}\Bigg]
=\Bigg[\Bigg[\tree{a_i,a_j,b_2},\tree{a_2,a_k,b_1}\Bigg],\tree{a_{1},b_{2},b_{3}}\Bigg],
\]
with $i,j,k\notin \lbrace 1,2,3 \rbrace$, which by construction belongs to $\Gamma_3(W(a^{\geq 1}b^{\geq 1}))$.

\textbf{Four contractions.} If there are exactly $4$ contractions between
$\tree{a_i,a_j,a_k,a_l}$ and $\tree{b_{1},b_{2},b_{3}}$, modulo the AS relation, we are in the $\mathfrak{S}_g$-orbits of the elements:
\begin{equation}
\label{eq:lemma-4contract-list}
\Bigg[\tree{a_1,a_2,a_1,a_3},\tree{b_{1},b_{2},b_{3}}\Bigg]
\quad \text{and} \quad
\Bigg[\tree{a_1,a_2,a_1,a_2},\tree{b_{1},b_{2},b_{3}}\Bigg].
\end{equation}

We show that the first element of \eqref{eq:lemma-4contract-list} belongs to $\Gamma_3(W(a^{\geq 1}b^{\geq 1}))$. Consider the first element of \eqref{eq:lemma-3contract-list}, which belongs to $\Gamma_3(W(a^{\geq 1}b^{\geq 1}))$ by the three contractions case. Act by $G\in GL_g(\mathbb{Z})$, which sends $a_4$ to $a_4+a_1$ and $b_1$ to $b_1-b_4$. Then we compute:
\[
\begin{aligned}
G\Bigg(\Bigg[\tree{a_1,a_2,a_4,a_3},\tree{b_{1},b_{2},b_{3}}\Bigg]\Bigg)
= \; & \Bigg[\tree{a_1,a_2,a_4,a_3},\tree{b_{1},b_{2},b_{3}}\Bigg]  
+\Bigg[\tree{a_1,a_2,a_1,a_3},\tree{b_{1},b_{2},b_{3}}\Bigg] \\
& -\Bigg[\tree{a_1,a_2,a_4,a_3},\tree{b_{4},b_{2},b_{3}}\Bigg]
-\Bigg[\tree{a_1,a_2,a_1,a_3},\tree{b_{4},b_{2},b_{3}}\Bigg].
\end{aligned}
\]
Then, by the two and three contractions cases, the first, third and fourth brackets above belong to $\Gamma_3(W(a^{\geq 1}b^{\geq 1}))$ and hence,
\[
\Bigg[\tree{a_1,a_2,a_1,a_3},\tree{b_{1},b_{2},b_{3}}\Bigg]\in \Gamma_3(W(a^{\geq 1}b^{\geq 1})).
\]

Finally, we prove that the second element of \eqref{eq:lemma-4contract-list} belongs to $\Gamma_3(W(a^{\geq 1}b^{\geq 1}))$.
Consider the first element of \eqref{eq:lemma-4contract-list} which we already know that belongs to $\Gamma_3(W(a^{\geq 1}b^{\geq 1}))$. Act by $G\in GL_g(\mathbb{Z})$, which sends $a_3$ to $a_3+a_2$ and $b_2$ to $b_2-b_3$. Then we compute:
\[
G\Bigg(\Bigg[\tree{a_1,a_2,a_1,a_3},\tree{b_{1},b_{2},b_{3}}\Bigg]\Bigg)= \Bigg[\tree{a_1,a_2,a_1,a_3},\tree{b_{1},b_{2},b_{3}}\Bigg]  +\Bigg[\tree{a_1,a_2,a_1,a_2},\tree{b_{1},b_{2},b_{3}}\Bigg],
\]
where we used that, by linearity on the leaves and the AS relation, the action of the element $G$ leaves the second tree of the first bracket invariant.
Therefore we get that the second element of \eqref{eq:lemma-4contract-list} also belongs to $\Gamma_3(W(a^{\geq 1}b^{\geq 1}))$.
\end{proof}

\begin{lemma}
\label{lema:[[A,AAB],AAB]}
For any integer $g\geq 4$ we have the following inclusion of modules:
\[
[[W(a^3), W(ab^2)], W(ab^2)]\subset \Gamma_3(W(a^{\geq 1}b^{\geq 1})).
\]
\end{lemma}

\begin{proof}
By Lemma \ref{lema:[A,ABB]}, the module $[[W(a^3), W(ab^2)], W(ab^2)]$, modulo $\Gamma_3(W(a^{\geq 1}b^{\geq 1}))$, is generated by the $\mathfrak{S}_g$-orbits of the elements:
\[
\Bigg[2\tree{a_i,a_1,b_{1},a_j}, \tree{a_k,b_{i'},b_{j'}}\Bigg].
\]

We first notice that without loss of generality we can assume that $i',j',k\neq 1$. Because 
if one of the coefficients $i',j',k$ is $1$ then there is an integer $l\neq i,j,i',j',k$ and then,
\[
\tree{a_j,a_l,b_{l},a_i}-\tree{a_i,a_1,b_{1},a_j}=\Bigg[\tree{a_i,a_1,b_l},\tree{a_l,b_{1},a_j}\Bigg] \in \Gamma_2(W(a^{\geq 1}b^{\geq 1})).
\]
Hence, modulo $\Gamma_3(W(a^{\geq 1}b^{\geq 1}))$, we have that
\[
\Bigg[2\tree{a_i,a_1,b_{1},a_j}, \tree{a_k,b_{i'},b_{j'}}\Bigg]=\Bigg[2\tree{a_j,a_l,b_{l},a_i}, \tree{a_k,b_{i'},b_{j'}}\Bigg].
\]
Therefore the module $[[W(a^3), W(ab^2)], W(ab^2)]$, modulo $\Gamma_3(W(a^{\geq 1}b^{\geq 1}))$, is generated by the $\mathfrak{S}_g$-orbit of the element:
\[
\Bigg[2\tree{a_i,a_1,b_{1},a_j}, \tree{a_k,b_{2},b_{3}}\Bigg] \quad  \text{with} \;i,j,k\neq 1.
\]

We now prove that each of these generators belongs to $\Gamma_3(W(a^{\geq 1}b^{\geq 1}))$. We proceed by distinguishing the number of contractions between the trees of the Lie bracket.

\textbf{One contraction.} If there is a single contraction between $\tree{a_i,a_1,b_{1},a_j}$ and 
$\tree{a_k,b_{2},b_{3}}$, we are in the $\mathfrak{S}_g$-orbit of the element:
\[
\Bigg[ 2\tree{a_2,a_1,b_{1},a_4}, \tree{a_k,b_{2},b_{3}} \Bigg]= -\Bigg[ 2\tree{b_2,a_1,b_{1},a_4}, \tree{a_k,a_{2},b_{3}} \Bigg]\quad \text{with} \;k\neq 1,
\]
which belongs to $\Gamma_3(W(a^{\geq 1}b^{\geq 1}))$ by Lemma \ref{lema:AABB-elements}.

\textbf{Two contractions.} If there are exactly $2$ contractions between $\tree{a_i,a_1,b_{1},a_j}$ and 
$\tree{a_k,b_{i'},b_{j'}}$, we are in the $\mathfrak{S}_g$-orbits of the elements:
\begin{equation}
\label{eq:2-contract-elements}
\Bigg[ 2\tree{a_2,a_1,b_{1},a_4}, \tree{a_k,b_{2},b_{4}} \Bigg] \quad \text{and}\quad
\Bigg[ 2\tree{a_2,a_1,b_{1},a_2}, \tree{a_k,b_{2},b_{3}}\Bigg],
\qquad \text{with} \;k\neq 1.
\end{equation}

We show that the first element of \eqref{eq:2-contract-elements} belongs to $\Gamma_3(W(a^{\geq 1}b^{\geq 1}))$. We compute:
\begin{equation}
\label{eq:2-contract}
\Bigg[ 2\tree{a_2,a_1,b_{1},a_4}, \tree{a_k,b_{2},b_{4}} \Bigg] =
2\tree{a_4,b_1,a_{1},b_{4},a_k}-2\tree{a_2,a_1,b_{1},b_2,a_k}.
\end{equation}

Consider the following sum of trees, which by Lemma \ref{lema:AABB-elements} belongs to $\Gamma_3(W(a^{\geq 1}b^{\geq 1}))$.
\[
\Bigg[ 2\tree{a_2,a_1,b_{1},b_4}, \tree{a_k,b_{2},a_{4}} \Bigg]+\Bigg[ 2\tree{b_2,a_1,b_{1},a_4}, \tree{a_k,a_{2},b_{4}} \Bigg]
+\Bigg[\tree{b_1,a_1,a_3},2\tree{b_3,b_2,a_2,a_k}\Bigg].
\]
By construction, this sum is equal to:
\[
\begin{aligned}
& 2\tree{b_4,b_1,a_{1},a_{4},a_k}-2\tree{a_2,a_1,b_{1},a_k,b_2}
-2\tree{a_4,b_1,a_{1},b_{4},a_k}+2\tree{b_2,a_1,b_{1},a_k,a_2}
+2\tree{b_1,a_1,b_2,a_2,a_k} \\
& = 2\tree{b_4,b_1,a_{1},a_{4},a_k}+2\tree{a_2,a_1,b_{1},b_2,a_k}
-2\tree{a_4,b_1,a_{1},b_{4},a_k}-2\tree{b_2,a_1,b_{1},a_2,a_k}
+2\tree{b_1,a_1,b_2,a_2,a_k} \\
& = 2\tree{b_4,b_1,a_{1},a_{4},a_k}+2\tree{a_2,a_1,b_{1},b_2,a_k}
-2\tree{a_4,b_1,a_{1},b_{4},a_k}+2\tree{b_1,b_2,a_1,a_2,a_k},
\end{aligned}
\]
where in the first equality we used the AS relation in the second and fourth trees, and in the second equality we used the IHX relation in the fifth tree.

Adding this element to \eqref{eq:2-contract} we get that, modulo $\Gamma_3(W(a^{\geq 1}b^{\geq 1}))$,
\[
\Bigg[2 \tree{a_2,a_1,b_{1},a_4}, \tree{a_k,b_{2},b_{4}} \Bigg] =
2\tree{b_4,b_1,a_{1},a_{4},a_k} +2\tree{b_1,b_2,a_1,a_2,a_k} =
-2\tree{a_k,a_4,a_1,b_1,b_4}+2\tree{a_k,a_2,a_1,b_1,b_2},
\]
where in the second equality we used the AS relation in both trees.
Then, by point $iii)$ in Lemma \ref{lema:trees-a4b2-needed}, this element belongs to
$\Gamma_3(W(a^{\geq 1}b^{\geq 1}))$.

Finally, we show that the second element of \eqref{eq:2-contract-elements} belongs to $\Gamma_3(W(a^{\geq 1}b^{\geq 1}))$.
Consider the first element of \eqref{eq:2-contract-elements}, which we already know that belongs to $\Gamma_3(W(a^{\geq 1}b^{\geq 1}))$. Act by $G\in GL_g(\mathbb{Z})$, which sends $a_4$ to $a_4+a_2$ and $b_2$ to $b_2-b_4$. Then we compute:
\begin{equation*}
G\Bigg(2\Bigg[ \tree{a_2,a_1,b_{1},a_4}, \tree{a_k,b_{2},b_{4}} \Bigg]\Bigg)=
2\Bigg[ \tree{a_2,a_1,b_{1},a_4}, \tree{a_k,b_{2},b_{4}} \Bigg]
+2\Bigg[ \tree{a_2,a_1,b_{1},a_2}, \tree{a_k,b_{2},b_{4}} \Bigg],
\end{equation*}
where we used that, by linearity on the leaves and the AS relation, the action of the element $G$ leaves the second tree of the first bracket invariant.
This shows that the second element of \eqref{eq:2-contract-elements} also belongs to 
$\Gamma_3(W(a^{\geq 1}b^{\geq 1}))$.
\end{proof}
\vspace{0.3cm}

\textbf{The submodule $\overline{W}(a^4b)$.} By the same argument as before, the submodule $\overline{W}(a^4b)$ is generated by the image of the triple Lie bracket $[[\;,\;],\;]$ on the following modules:
\[
\begin{array}{ll}
(W(a^3)\wedge W(b^3))\otimes W(a^3), & (W(a^3)\wedge W(a^2b))\otimes W(ab^2), \\[1ex]
(W(a^3)\wedge W(ab^2))\otimes W(a^2b), & (W(a^2b)\wedge W(ab^2))\otimes W(a^3), \\[1ex]
\multicolumn{2}{c}{(W(a^2b)\wedge W(a^2b))\otimes W(a^2b).}
\end{array}
\]

We further reduce this list of modules as follows.
By Lemma \ref{lema:[A,B]}, the image of the module $(W(a^3)\wedge W(b^3))\otimes W(a^3)$ by the triple Lie bracket is contained in the image of the module
$(W(ab^2)\wedge W(a^2b))\otimes W(a^3)$ and, by the Jacobi identity, the image of this last module is contained in the sum of the images of the modules $(W(a^3)\wedge W(a^2b))\otimes W(ab^2)$ and $(W(a^3)\wedge W(ab^2))\otimes W(a^2b)$.
Therefore, the module $\overline{W}(a^4b)$ is generated by the image of the triple Lie bracket on the following modules:
\begin{equation}
\label{eq:list-mod-a4b1}
\begin{array}{ll}
 (W(a^3)\wedge W(a^2b))\otimes W(ab^2), & (W(a^3)\wedge W(ab^2))\otimes W(a^2b), \\[1ex]
 \multicolumn{2}{c}{(W(a^2b)\wedge W(a^2b))\otimes W(a^2b).}  
\end{array}
\end{equation}

Before computing the image of these modules by the triple Lie bracket $[[\;,\;],\;]$, modulo $\Gamma_3(W(a^2b))$, we show the following preliminary result:

\begin{lemma}
\label{lema:trees-ab4-needed}
For any integer $g \geq 6$ the following trees belong to $\Gamma_3(W(a^2b))$.
\begin{multicols}{2}
\begin{enumerate}[i)]
\item $\tree{a_3,a_2,a_j,b_1,a_i}$ for
$j\neq 1$,
\item $\tree{a_i,a_1,a_j,b_1,a_k}-\tree{a_i,a_{2},a_j,b_{2},a_k}$
$\begin{array}{c}
\text{for} \\
i,j\notin \lbrace 1,2 \rbrace,
\end{array}$
\item $\tree{a_i,a_j,a_1,b_1,a_k}-\tree{a_i,a_j,a_2,b_2,a_k}$
$\begin{array}{c}
\text{for} \\
i,j\notin \lbrace 1,2 \rbrace,
\end{array}$
\item $\tree{a_i,a_1,a_j,b_1,a_k}-\tree{a_k,a_1,a_i,b_1,a_j}$
$\begin{array}{c}
\text{for} \\
i,j,k \neq 1,
\end{array}
$
\item $\tree{a_i,a_j,a_1,b_1,a_k}-\tree{a_k,a_i,a_1,b_1,a_j}$
$\begin{array}{c}
\text{for} \\
i,j,k \neq 1,
\end{array}$
\item $\tree{a_3,a_1,a_1,b_1,a_k}-\tree{a_3,a_1,a_{2},b_2,a_k}
- \tree{a_3,a_2,a_1,b_2,a_k}$ for $k\neq 2$. 
\end{enumerate}
\end{multicols}

\end{lemma}

\begin{proof}

\textbf{Elements i)} Let $i,j\in \mathbb{N}$ such that $j\neq 1$. Given $l\in \mathbb{N}$ such that $l\notin \lbrace i,j,1,2,3 \rbrace$. A direct computation shows that:
\[
\tree{a_3,a_2,a_j,b_1,a_i}=\Bigg[\Bigg[\tree{a_3,a_2,b_{l}},\tree{a_{l},a_j,b_{l}}\Bigg],
\tree{a_{l},b_1,a_i}\Bigg].
\]
Therefore this element belongs to $\Gamma_3(W(a^2b))$.
\vspace{0.5cm}

\textbf{Elements ii)}
Let $i,j,k\in \mathbb{N}$ such that $i,j\neq 1$ and $i,j,k\neq 2$. Consider the element $\tree{a_i,a_2,a_j,b_1,a_k},$ which belongs to $\Gamma_3(W(a^2b))$ by point $i)$. Act by $G\in GL_g(\mathbb{Z})$, which sends $b_{1}$ to $b_{1}-b_{2}$ and $a_{2}$ to $a_{2}+a_{1}$. Then we get another element of $\Gamma_3(W(a^2b))$ given by:
\[
G\Bigg(\tree{a_i,a_2,a_j,b_1,a_k}\Bigg)= \tree{a_i,a_2,a_j,b_1,a_k}+\tree{a_i,a_1,a_j,b_{1},a_k} -\tree{a_i,a_2,a_j,b_{2},a_k}-\tree{a_i,a_1,a_j,b_2,a_k},
\]
which, by point $i)$, modulo $\Gamma_3(W(a^2b))$, is equal to:
\[
\tree{a_i,a_1,a_j,b_{1},a_k}-\tree{a_i,a_2,a_j,b_{2},a_k}
\quad \text{with } i,j\notin \lbrace 1,2 \rbrace \text{ and } k\neq 2.
\]
Therefore this element belongs to $\Gamma_3(W(a^2b))$.

Notice that, in fact, we can drop out the condition $k\neq 2$.
Because, if we act by the permutation $G=(1,2)$ on the element $\tree{a_i,a_2,a_j,b_{2},a_1}-\tree{a_i,a_1,a_j,b_{1},a_1}$, which we already know that belongs to $\Gamma_3(W(a^2b))$, we have the following equality:
\[
G\Bigg(\tree{a_i,a_2,a_j,b_{2},a_1}-\tree{a_i,a_1,a_j,b_{1},a_1}\Bigg)
=\tree{a_i,a_1,a_j,b_{1},a_2}-\tree{a_i,a_2,a_j,b_{2},a_2}.
\]
Hence this last element also belongs to $\Gamma_3(W(a^2b))$.
\vspace{0.5cm}

\textbf{Elements iii)}
Let $i,j,k\in \mathbb{N}$ such that $i,j\notin \lbrace 1,2 \rbrace$. Using the IHX and AS relations, we get that,
\[
\begin{aligned}
\tree{a_i,a_j,a_1,b_1,a_k}-\tree{a_i,a_j,a_2,b_2,a_k} & =
\tree{a_i,a_1,a_j,b_1,a_k}+\tree{a_1,a_j,a_i,b_1,a_k}
-\tree{a_i,a_2,a_j,b_2,a_k}-\tree{a_2,a_j,a_i,b_2,a_k} \\ & =
\tree{a_i,a_1,a_j,b_1,a_k}-\tree{a_j,a_1,a_i,b_1,a_k}
-\tree{a_i,a_2,a_j,b_2,a_k}+\tree{a_j,a_2,a_i,b_2,a_k},
\end{aligned}
\]
and by point $ii)$ this element belongs to $\Gamma_3(W(a^2b))$.

\vspace{0.5cm}

\textbf{Elements iv)}
Let $i,j,k\in \mathbb{N}$ such that $i,j,k\neq 1$. Given $l\in \mathbb{N}$ such that $l\notin \lbrace 1,i,j,k \rbrace$, a direct computation shows that:
\[
\begin{aligned}
\Bigg[ \Bigg[\tree{a_i,a_1,b_{l}},\tree{a_{l},b_1,a_j}\Bigg], \tree{a_1,a_k,b_1}\Bigg]= &
\; \Bigg[ \tree{b_{l},a_i,a_j,a_{l}}, \tree{a_1,a_k,b_1}\Bigg]
 -\Bigg[ \tree{a_i,a_1,b_1,a_j}, \tree{a_1,a_k,b_1}\Bigg] \\
= & \; \tree{a_1,a_i,a_j,a_k,b_1}-\tree{b_1,a_j,a_i,a_1,a_k} \\
= & \; \tree{a_i,a_1,a_j,b_1,a_k}-\tree{a_k,a_1,a_i,b_1,a_j},
\end{aligned}
\]
where in the last equality we used the AS relation on both trees.
Therefore this element belongs to $\Gamma_3(W(a^2b))$.
\vspace{0.5cm}

\textbf{Elements v)}
Let $i,j,k\in \mathbb{N}$ such that $i,j,k\neq 1$.
Using the IHX relation and after the AS relation, we get that,
\[
\begin{aligned}
\tree{a_i,a_j,a_1,b_1,a_k}-\tree{a_k,a_i,a_1,b_1,a_j} & =
\tree{a_i,a_1,a_j,b_1,a_k}+\tree{a_1,a_j,a_i,b_1,a_k}
-\tree{a_k,a_1,a_i,b_1,a_j}-\tree{a_1,a_i,a_k,b_1,a_j} \\
& =
\tree{a_i,a_1,a_j,b_1,a_k}-\tree{a_j,a_1,a_i,b_1,a_k}
-\tree{a_k,a_1,a_i,b_1,a_j}+\tree{a_i,a_1,a_k,b_1,a_j},
\end{aligned}
\]
and by point $iv)$ this element belongs to $\Gamma_3(W(a^2b))$.
\vspace{0.5cm}

\textbf{Elements vi)}
Let $k\in \mathbb{N}$ such that k$\neq 2$. Given $l\in \mathbb{N}$ such that $l\notin \lbrace 1,2,3 \rbrace$. Consider the element of $\Gamma_3(W(a^2b))$,
\[
\begin{aligned}
\Bigg[ \Bigg[\tree{a_3,a_1,b_{l}},\tree{a_{l},a_1,b_2}\Bigg], \tree{a_2,b_1,a_k} \Bigg]& =
-\Bigg[  \tree{a_3,a_1,a_1,b_2},\tree{a_2,b_1,a_k}\Bigg] \\
 & = \tree{a_3,a_1,a_1,b_1,a_k}-\tree{a_1,a_{3},b_{2},a_k,a_2} - \tree{a_1,b_2,a_3,a_k,a_2} \\
 & = \tree{a_3,a_1,a_1,b_1,a_k}-\tree{a_3,a_1,b_{2},a_2,a_k} - \tree{a_k,a_2,a_3,b_2,a_1} \\
& = \tree{a_3,a_1,a_1,b_1,a_k}-\tree{a_3,a_{1},a_{2},b_2,a_k}-\tree{a_3,a_{1},a_k,a_2,b_2} - \tree{a_k,a_2,a_3,b_2,a_1},
\end{aligned}
\]
where in the third equality we used the AS relation in the last two trees, and in the last equality we used the IHX relation on the second tree.

Finally, by points i) and iv), modulo $\Gamma_3(W(a^2b))$, this element is equal to:
\[
\tree{a_3,a_1,a_1,b_1,a_k}-\tree{a_3,a_1,a_{2},b_2,a_k}
- \tree{a_3,a_2,a_1,b_2,a_k} \quad \text{for } k\neq 2. 
\]
Therefore this element belongs to $\Gamma_3(W(a^2b))$.
\end{proof}

We now show that the image of the module $(W(a^3)\wedge W(ab^2))\otimes W(a^2b)$ by the triple Lie bracket is contained in the sum of the image of the modules $(W(a^3)\wedge W(a^2b))\otimes W(ab^2)$ and $(W(a^2b)\wedge W(a^2b))\otimes W(a^2b)$. As a consequence, the submodule $\overline{W}(a^4b^1)$ is the sum of modules $[[W(a^3), W(a^2b)],W(ab^2)]$ and $\Gamma_3(W(a^2b))$.

\begin{lemma}
\label{lema:gen W_1ab4}
For any integer $g \geq 6$, we have the following inclusion:
\[
[[W(a^3), W(ab^2)], W(a^2b)] \subset [[W(a^3), W(a^2b)],W(ab^2)]+\Gamma_3(W(a^2b)).
\]
Moreover, the module $[[W(a^3), W(a^2b)],W(ab^2)]$, modulo $\Gamma_3(W(a^2b))$, is generated by the $\mathfrak{S}_g$-orbit of the tree:
\[
\tree{a_1,a_2,a_4,b_4,a_3}.
\]
\end{lemma}

\begin{proof}
We first prove that
the module $[[W(a^3), W(a^2b)],W(ab^2)]$, modulo $\Gamma_3(W(a^2b))$, is generated by
the $\mathfrak{S}_g$-orbit of the tree
$\tree{a_1,a_2,a_4,b_4,a_3}$. Then we show that the module 
$[[W(a^3), W(ab^2)], W(a^2b)]$, modulo $\Gamma_3(W(a^2b))$, is generated by
the $\mathfrak{S}_g$-orbit of the tree $2\tree{a_1,a_2,a_4,b_4,a_3}$. Therefore, modulo $\Gamma_3(W(a^2b))$, the module
$[[W(a^3), W(ab^2)], W(a^2b)]$ is contained in $[[W(a^3), W(a^2b)],W(ab^2)],$
and hence we get the inclusion of the statement.

Remember that, as we have already seen at the beginning of the proof of Lemma \ref{lema:[[A,AB],B]}, we know that
$[W(a^3), W(a^2b)]= W(a^4)$. Hence,
\[
[[W(a^3), W(a^2b)],W(ab^2)]=[W(a^4),W(ab^2)].
\]
The generators of this last module are given by the $\mathfrak{S}_g$-orbits of the following elements:
\[
\Bigg[\tree{a_l,a_i,a_j,a_m}, \tree{b_{1},b_{2},a_{k}}\Bigg]\quad \text{with }
1\leq i,j,k,l,m \leq g.
\]

We now reduce the number of generators,
modulo $\Gamma_3(W(a^2b))$, in terms of the contractions between the first and second tree.
\vspace{0.3cm}

\textbf{One contraction.} If there is a single contraction between $\tree{a_l,a_i,a_j,a_m}$ and $\tree{b_{1},b_{2},a_{k}}$, modulo the AS relation, we are in the $\mathfrak{S}_g$-orbit of the element:
\[
\Bigg[\tree{a_3,a_4,a_j,a_1},\tree{b_{1},b_{2},a_k}\Bigg]=-\Bigg[\tree{a_3,a_4,a_j,b_1},\tree{a_{1},b_{2},a_k}\Bigg]
=\Bigg[\Bigg[\tree{a_3,a_4,b_2},\tree{a_2,a_j,b_1}\Bigg],\tree{a_{1},b_{2},a_k}\Bigg],
\]
with $j\neq 1,2$, which by construction belongs to $\Gamma_3(W(a^2b))$.
\vspace{0.3cm}

\textbf{Two contractions.} If there are exactly $2$ contractions between $\tree{a_l,a_i,a_j,a_m}$ and $\tree{b_{1},b_{2},a_{k}}$, modulo the AS relation, we are in the $\mathfrak{S}_g$-orbits of the elements:
\begin{equation}
\label{eq:last-2-contract-elements}
\Bigg[\tree{a_1,a_i,a_j,a_1}, \tree{b_1,b_2,a_k}\Bigg], \qquad
\Bigg[\tree{a_1,a_i,a_j,a_2}, \tree{b_2,b_1,a_k}\Bigg] \quad \text{and} \quad
\Bigg[\tree{a_1,a_2,a_3,a_4}, \tree{b_1,b_2,a_k}\Bigg],
\end{equation}
with $i,j\notin \lbrace 1,2 \rbrace$.

For the first element of \eqref{eq:last-2-contract-elements} we compute:
\[
\Bigg[\tree{a_1,a_i,a_j,a_1}, \tree{b_1,b_2,a_k}\Bigg] = \tree{a_1,a_i,a_j,b_2,a_k}+\tree{a_1,a_j,a_i,b_2,a_k}
\quad \text{with } i,j\notin \lbrace 1,2 \rbrace,
\]
which by point $i)$ in Lemma \ref{lema:trees-ab4-needed} belongs to $\Gamma_3(W(a^2b))$.

For the second element of \eqref{eq:last-2-contract-elements} we compute:
\[
\begin{aligned}
\Bigg[\tree{a_1,a_i,a_j,a_2}, \tree{b_2,b_1,a_k}\Bigg] & = \tree{a_1,a_i,a_j,b_1,a_k}+\tree{a_2,a_j,a_i,a_k,b_2} \\
& = \tree{a_j,a_i,a_1,b_1,a_k}+\tree{a_1,a_j,a_i,b_1,a_k}
+\tree{a_2,a_j,a_i,a_k,b_2}\\
& = \tree{a_j,a_i,a_1,b_1,a_k}-\tree{a_j,a_1,a_i,b_1,a_k}
+\tree{a_j,a_2,a_i,b_2,a_k} \quad \text{with } i,j\notin \lbrace 1,2 \rbrace,
\end{aligned}
\]
where in the second equality we used the IHX relation on the first tree,
and in the last equality we used the AS relation on the last two trees.

By point $ii)$ in Lemma \ref{lema:trees-ab4-needed}, modulo
$\Gamma_3(W(a^2b))$, this element becomes:
\[
\tree{a_j,a_i,a_1,b_1,a_k} \quad \text{with } i,j\notin \lbrace 1,2 \rbrace.
\]

For the third element of \eqref{eq:last-2-contract-elements}, using the IHX relation we compute:
\[
\Bigg[\tree{a_1,a_2,a_3,a_4}, \tree{b_1,b_2,a_k}\Bigg]=
\Bigg[\tree{a_1,a_3,a_2,a_4}, \tree{b_1,b_2,a_k}\Bigg]+
\Bigg[\tree{a_3,a_2,a_1,a_4}, \tree{b_1,b_2,a_k}\Bigg],
\]
which is a sum of the second type of elements \eqref{eq:last-2-contract-elements}.
\vspace{0.3cm}

\textbf{Three contractions.} If there are exactly $3$ contractions between $\tree{a_l,a_i,a_j,a_m}$ and $\tree{b_{1},b_{2},a_{k}}$, modulo the AS relation, we are in the $\mathfrak{S}_g$-orbits of the element:
\[
\begin{aligned}
\Bigg[\tree{a_1,a_2,a_3,a_1}, \tree{b_1,b_2,a_k}\Bigg]
& = \tree{a_1,a_2,a_3,b_2,a_k}+\tree{a_1,a_3,a_2,b_2,a_k}
-\tree{a_3,a_1,a_1,b_1,a_k},
\end{aligned}
\]
where in the last equality we used the AS relation on the third tree.

Without loss of generality, we may assume that $k\neq 4$.
By point $vi)$ in Lemma \ref{lema:trees-ab4-needed}, modulo
$\Gamma_3(W(a^2b))$, this element is equal to:
\[
\tree{a_1,a_2,a_3,b_2,a_k}+\tree{a_1,a_3,a_2,b_2,a_k} -\tree{a_3,a_4,a_1,b_4,a_k}
-\tree{a_3,a_1,a_4,b_4,a_k},
\]
and, by point $ii)$ in Lemma \ref{lema:trees-ab4-needed}, modulo
$\Gamma_3(W(a^2b))$, this element becomes:
\[
\tree{a_1,a_2,a_3,b_2,a_k}+\tree{a_1,a_3,a_2,b_2,a_k} -\tree{a_3,a_2,a_1,b_2,a_k}
-\tree{a_3,a_1,a_4,b_4,a_k},
\]
which, by the IHX relation on the first tree, is equal to:
\[
2\tree{a_1,a_3,a_2,b_2,a_k}
-\tree{a_3,a_1,a_4,b_4,a_k}.
\]

\textbf{Four contractions.} If there are exactly $4$ contractions between $\tree{a_l,a_i,a_j,a_m}$ and $\tree{b_{1},b_{2},a_{k}}$, modulo the AS relation, we are in the $\mathfrak{S}_g$-orbit of the element:
\[
\Bigg[\tree{a_1,a_2,a_2,a_1}, \tree{b_1,b_2,a_k}\Bigg]=2\tree{a_1,a_2,a_2,b_2,a_k}+2\tree{a_2,a_1,a_1,a_k,b_1}
=2\tree{a_1,a_2,a_2,b_2,a_k}-2\tree{a_2,a_1,a_1,b_1,a_k},
\]
where in the last equality we used the AS relation on the second tree.
Without loss of generality, we may assume that $k\neq 3$.

By point $vi)$ in Lemma \ref{lema:trees-ab4-needed}, modulo
$\Gamma_3(W(a^2b))$, this element is equal to:
\[
2\tree{a_1,a_3,a_2,b_3,a_k}-2\tree{a_2,a_3,a_1,b_3,a_k}
+2\tree{a_1,a_2,a_3,b_3,a_k}
-2\tree{a_2,a_1,a_3,b_3,a_k}
=6\tree{a_1,a_2,a_3,b_3,a_k},
\]
where in the last equality we used the 
IHX relation on the first tree and the AS relation on the fourth tree.

\vspace{0.3cm}

Therefore the generators of the module $[[W(a^3), W(a^2b)], W(ab^2)]$, modulo
$\Gamma_3(W(a^2b))$, are the $\mathfrak{S}_g$-orbits of the elements of the form:
\[
\tree{a_1,a_2,a_4,b_4,a_k} \quad \text{with } k\in \mathbb{N},
\]
which are the $\mathfrak{S}_g$-orbits of the following elements:
\[
\tree{a_1,a_2,a_4,b_4,a_1},\quad \tree{a_1,a_2,a_4,b_4,a_4}\quad \text{and} \quad \tree{a_1,a_2,a_4,b_4,a_3}.
\]
By point $v)$ in Lemma \ref{lema:trees-ab4-needed} and the AS relation, modulo $\Gamma_3(W(a^2b))$,
\begin{equation}
\label{eq:gen-AS-red}
\tree{a_1,a_2,a_4,b_4,a_1}=\tree{a_1,a_1,a_4,b_4,a_2}=0.
\end{equation}
Moreover, by point $ii)$ in Lemma \ref{lema:trees-ab4-needed}, modulo $\Gamma_3(W(a^2b))$,
\[
\tree{a_1,a_2,a_4,b_4,a_4}=\tree{a_1,a_2,a_5,b_5,a_4},
\]
and this last element is in the $\mathfrak{S}_g$-orbit of $\tree{a_1,a_2,a_4,b_4,a_3}$.

Therefore, the module $[[W(a^3), W(a^2b)], W(ab^2)]$, modulo $\Gamma_3(W(a^2b))$, is generated by the
$\mathfrak{S}_g$-orbit of the element:
\begin{equation}
\label{eq:gen-1-trees}
\tree{a_1,a_2,a_4,b_4,a_3}.
\end{equation}

We now reduce the set of generators of $[[W(a^3),W(ab^2)], W(a^2b)]$, modulo $\Gamma_3(W(a^2b))$.
By Lemma \ref{lema:[A,ABB]}, this module is generated by the elements:
\[
\Bigg[2\tree{a_i,b_4,a_4,a_{j}}, \tree{b_k,a_l,a_m}\Bigg] \quad \text{with}\quad i,j,k,l,m\in \mathbb{N}.
\]
In fact, analogously to the first part of the proof of Lemma \ref{lema:[[A,AAB],AAB]}, without loss of generality, we can assume that $i,j,k,l,m\neq 4$. Therefore we have the $\mathfrak{S}_g$-orbits of the elements:
\[
\Bigg[2\tree{a_{i},b_4,a_4,a_j}, \tree{b_k,a_2,a_1}\Bigg] \quad \text{with}\quad i,j,k\neq 4.
\]

Then the image of the module, modulo $\Gamma_3(W(a^2b))$, is generated by the $\mathfrak{S}_g$-orbits of the image of the elements:
\[
\begin{array}{c}
\Bigg[2\tree{a_i,b_4,a_4,a_j}, \tree{b_j,a_2,a_1}\Bigg]=2\tree{a_i,b_4,a_4,a_2,a_1}, \qquad
\Bigg[2\tree{a_j,b_4,a_4,a_i}, \tree{b_j,a_2,a_1}\Bigg]=2\tree{a_i,a_4,b_4,a_2,a_1}, \\
\Bigg[2\tree{a_i,b_4,a_4,a_i}, \tree{b_i,a_2,a_1}\Bigg] =
2\tree{a_i,b_4,a_4,a_2,a_1}+2\tree{a_i,a_4,b_4,a_2,a_1},
\end{array}
\]
with $i,j\neq 4$ and $i\neq j$.

Notice that the third element is the sum of the first two elements.
Moreover, by the IHX relation we have that that
\[
2\tree{a_i,b_4,a_4,a_2,a_1}=
2\tree{a_i,a_4,b_4,a_2,a_1}+2\tree{a_4,b_4,a_i,a_2,a_1} \quad \text{with }
i\neq 4,
\]
where the last summand belongs to $\Gamma_3(W(a^2b))$ by point $i)$ of Lemma \ref{lema:trees-ab4-needed}.

Therefore the generators of the module $[[W(a^3), W(ab^2)], W(a^2b)]$, modulo $\Gamma_3(W(a^2b))$, are the $\mathfrak{S}_g$-orbits of the elements:
\[
2\tree{a_i,b_4,a_4,a_2,a_1}=2\tree{a_1,a_2,a_4,b_4,a_i} \quad \text{with }
i\neq 4,
\]
where in the equality we used the AS relation.
By the equality \eqref{eq:gen-AS-red}, these trees are zero modulo $\Gamma_3(W(a^2b))$ unless $i\neq 1,2$. Hence, 
the module $[[W(a^3), W(ab^2)], W(a^2b)]$, modulo $\Gamma_3(W(a^2b))$, is generated by the $\mathfrak{S}_g$-orbit of the element:
\begin{equation}
\label{eq:gen-2-trees}
2\tree{a_1,a_2,a_4,b_4,a_3}.
\end{equation}

Finally, by looking at the generators \eqref{eq:gen-1-trees} and \eqref{eq:gen-2-trees},
modulo $\Gamma_3(W(a^2b))$, the module 
$[[W(a^3), W(ab^2)], W(a^2b)]$ is contained in the module
$[[W(a^3), W(a^2b)],W(ab^2)]$ and hence we get the inclusion of the statement.
\end{proof}

\textbf{The submodule $\overline{W}(a^5)$.} Finally, in the same way we showed $W(a^4)=[W(a^3),W(a^2b)]$ at the beginning of the proof of Lemma \ref{lema:[[A,AB],B]}, we show that $\overline{W}(a^5)=W(a^5)$ and that this submodule is equal to the image of the triple Lie bracket $[[\;,\;],\;]$ on the module:
\[
(W(a^3)\wedge W(a^2b))\otimes W(a^2b).
\]

By definition $\overline{W}(a^5)\subset W(a^5)$. Moreover any tree of $W(a^5)$ is in the $\mathfrak{S}_g$-orbit of:
\[
\tree{a_1,a_2,a_i,a_j,a_k}=\Bigg[\Bigg[\tree{a_1,a_2,a_{l}},\tree{b_{l},a_i,a_{l}}\Bigg],\tree{b_{l},a_j,a_k}\Bigg]
\quad \text{with }
\begin{array}{c}
1\leq i,j,k,l\leq g,\\
l\neq 1,2,i,j,k
\end{array}
\]
which belongs to $[[W(a^3), W(a^2b)], W(a^2b)]$.
Therefore we have the following equalities:
\[
\overline{W}(a^5)=W(a^5)=[[W(a^3), W(a^2b)], W(a^2b)].
\]

\vspace{0.3cm}

To sum up, looking at the coloring of leaves of the resulting trees when applying the triple Lie bracket $[[\;,\;],\;]$, we get that the submodules $\overline{W}(a^ib^j)$ given in the decomposition in Proposition \ref{prop:desc-m(3)}, are generated by the image of the triple Lie bracket $[[\;,\;],\;]$ on the following modules:

\renewcommand{\arraystretch}{1.3} 

\begin{table}[!ht]
\centering
\begin{tabular}{|C|C|}
\hline
\overline{W}(a^5) & (W(a^3)\wedge W(a^2b))\otimes W(a^2b),  \\ \hline 
\overline{W}(a^4b^1) & (W(a^3)\wedge W(a^2b))\otimes W(ab^2), \quad (W(a^2b)\wedge W(a^2b))\otimes W(a^2b),  \\ \hline 
\overline{W}(a^3b^2) & (W(a^2b)\wedge W(ab^2))\otimes W(a^2b),  \\ \hline
\overline{W}(a^2b^3) & (W(ab^2)\wedge W(a^2b))\otimes W(ab^2),  \\ \hline
\overline{W}(a^1b^4) & (W(b^3)\wedge W(ab^2))\otimes W(a^2b),\quad (W(ab^2)\wedge W(ab^2))\otimes W(ab^2), \\ \hline
\overline{W}(b^5) &  (W(b^3)\wedge W(ab^2))\otimes W(ab^2). \\ \hline 
\end{tabular}
\vspace{0.3cm}

\caption{Modules generating $Im(\tau_3)$.}
\label{table:1}
\end{table}

\renewcommand{\arraystretch}{1.0} 

\section{The image of the handlebody subgroups}

In this section we give a modified version of the Lagrangian trace map $Tr^A_k$ introduced by Faes in \cite[Sec. 4]{faes1}. Then using this trace map together with Table \ref{table:1} we compute the image of the handlebody subgroups $\mathcal{A}_{g,1}(3)$,  $\mathcal{B}_{g,1}(3)$ and $\mathcal{AB}_{g,1}(3)$ by the third Johnson homomorphism $\tau_3$.
Finally, we give a proof of Theorem~\ref{thm:equiv-rel-4}, which is the core result of this work.

\subsection{Lagrangian trace maps}
\label{sec:Lag-trace-maps}

In \cite[Sec. 4]{faes1} Faes introduced a Lagrangian trace map $Tr^A_k$ and in \cite[Thm. A]{faes3} he proved that his map vanishes on $\tau_k(\mathcal{A}_{g,1}(k))$. In particular, in \cite[Thm. 5.1]{faes1} he proved that for $k=2$ the intersection of the kernel of the Lagrangian trace map $Tr^A_2$ with the image of $\tau_2$ completely determines the group $\tau_2(\mathcal{A}_{g,1}(2))$.
We now make explicit his construction in our case of interest, $k=3$.

If we denote
$W_3(a^{\geq 1}b)$ the submodule of $\mathcal{A}_3(H)$ generated by those trees with at least one leaf colored with an $a$, we have a decomposition as $GL_g(\mathbb{Z})$-modules:
\begin{equation}
\label{eq:dec-trees-ab4}
W_3(a^{\geq 1}b)=W(ab^4)\oplus W(a^2b^3)\oplus W(a^3b^2)\oplus W(a^4b)\oplus W(a^5).
\end{equation}
The Lagrangian trace map of Faes given in \cite[Sec. 4]{faes1} takes the form:
\begin{equation}
\label{eq:trace-map-A-Faes}
\begin{tikzcd}[column sep=4mm]
Tr^A_3: W_3(a^{\geq 1}b) \ar[r,"\pi_{14}"] & W(ab^4) \ar[r,"\eta_3"] & A\otimes \mathcal{L}_4(B) \ar[r,"i"] & A\otimes (\otimes^4B) \ar[r,"\omega_{1,2}"] & \otimes^3B \ar[r] & S^3 B,
\end{tikzcd}
\end{equation}
where $\pi_{14}$ is the projection in the first component of the decomposition \eqref{eq:dec-trees-ab4}, $\eta_3$ is the expansion map defined in Section \ref{subsec-tree}, $i$ is the map defined recursively by sending $[c,d]$ to $c\otimes d-d\otimes c$, $\omega_{1,2}$ is the contraction between the first and second variable of $A\otimes (\otimes^4B)$, and the last map is the projection to the third symmetric product of $B$. Moreover, all these maps are $GL_g(\mathbb{Z})$-equivariant and hence the Lagrangian trace map $Tr^A_3$ is also $GL_g(\mathbb{Z})$-equivariant.

Explicitly, let $a \in A$ and $c,d,e,f \in H$, denote by $\overline{c}\in B$ the projection of $c\in H$ on $B$ parallel to $A$.
By the IHX and AS relations we have that
\[
\tree{c,d,a,e,f}=\tree{c,a,d,e,f}+\tree{a,d,c,e,f}
=\tree{c,a,d,e,f}-\tree{d,a,c,e,f}.
\]
Therefore any tree from $W_3(a^{\geq 1}b)$ is a sum of trees of the form $\tree{c,a,d,e,f}$.

On such trees, the Lagrangian trace map $Tr^A_3: W_3(a^{\geq 1}b)\rightarrow S^3B$ is given by:
\[
Tr^A_3\Big(\tree{c,a,d,e,f}\Big) = \omega(a,e)(\overline{c}\;\overline{d}\;\overline{f}) -\omega(a,f)(\overline{c}\;\overline{d}\;\overline{e}).
\]

Contrary to the case $k=2$, it turns out that the image of the third Johnson homomorphism $\tau_3$ and the Lagrangian trace map $Tr^A_3$ are not enough to determine $\tau_3(\mathcal{A}_{g,1}(3))$.
This is due to the fact that the map $Tr^A_3$ is zero on
$\overline{W}_3(a^{\geq 1}b)=W_3(a^{\geq 1}b)\cap Im(\tau_3)$ whereas, by Lemmas \ref{lema:Trace-map-A} and \ref{lema:ses-Trace-map-A}, there are trees in $\overline{W}_3(a^{\geq 1}b)$ that do not belong to $\tau_3(\mathcal{A}_{g,1}(3))$.

To prove the nullity of $Tr^A_3$ on $\overline{W}_3(a^{\geq 1}b)$ notice that by Lemma \ref{lema:gen W_1ab4} and Table \ref{table:1},
\[
\overline{W}_3(a^{\geq 1}b)=\overline{W}_3(a^{\geq 2}b)
+\Gamma_3(W(ab^2))+K,
\]
where $K$ is the submodule of $\overline{W}(ab^4)$ generated by the $\mathfrak{S}_g$-orbit of the tree $\tree{b_1,b_2,b_4,a_4,b_3}$.
Then,
a direct computation shows that the trace map $Tr^A_3$ vanishes on $\overline{W}_3(a^{\geq 2}b)$ and $K$. Moreover, by \cite[Prop. 6.2]{faes1} we know that $\tau_1(\mathcal{TAB}_{g,1}) =W(a^{\geq 1}b^{\geq 1})$ and hence, since $\tau_k$ is a Lie algebra homomorphism, we have inclusions
$\Gamma_3(W(ab^2))\subset \Gamma_3(\tau_1(\mathcal{TAB}_{g,1}))=
\tau_3(\mathcal{TAB}_{g,1}(3))\subset \tau_3(\mathcal{A}_{g,1}(3))$.
Then, by\cite[Thm. A]{faes3}, the map $Tr^A_3$ is zero on $\Gamma_3(W(ab^2))$.

To prove the existence of trees from $\overline{W}_3(a^{\geq 1}b)$ that do not belong to $\tau_3(\mathcal{A}_{g,1}(3))$, as well as to compute $\tau_3(\mathcal{A}_{g,1}(3))$,
we give a modified version of the Lagrangian trace map $Tr^A_3$ defined by Faes in \cite[Sec. 1]{faes3}. This modified version is obtained by taking the definition of the Lagrangian trace map $Tr_3^A$ and replacing the symmetric product $S^3B$ by the exterior product $\Lambda^3 B$. We call this modified trace: the antisymmetric Lagrangian trace map.
To be more precise,
we define the \textit{antisymmetric Lagrangian trace map} $Tr^A_\Lambda$ as the composition of maps:
\begin{equation}
\label{eq:trace-map-A}
\begin{tikzcd}[column sep=4mm]
Tr^A_\Lambda: W_3(a^{\geq 1}b) \ar[r,"\pi_{14}"] & W(ab^4) \ar[r,"\eta_3"] & A\otimes \mathcal{L}_4(B) \ar[r,"i"] & A\otimes (\otimes^4B) \ar[r,"\omega_{1,2}"] & \otimes^3B \ar[r] & \Lambda^3 B.
\end{tikzcd}
\end{equation}
Explicitly, let $a \in A$ and $c,d,e,f \in H$, denote by $\overline{c}\in B$ the projection of $c\in H$ on $B$ parallel to $A$.
Remember that, as before, by the AS and IHX relations, any tree of $W_3(a^{\geq 1}b)$ is a sum of trees of the form $\tree{c,a,d,e,f}$.

On these trees, the antisymmetric trace map $Tr^A_\Lambda: W_3(a^{\geq 1}b)\rightarrow \Lambda^3B$ is given by:
\[
Tr^A_\Lambda\Big(\tree{c,a,d,e,f}\Big) = 2\omega(a,d)(\overline{c}\wedge\overline{e}\wedge\overline{f})+ \omega(a,e)(\overline{c}\wedge\overline{d}\wedge\overline{f}) -\omega(a,f)(\overline{c}\wedge\overline{d}\wedge\overline{e}).
\]

We now prove that the subgroup $\tau_3(\mathcal{A}_{g,1}(3))$ is contained in the kernel of the antisymmetric Lagrangian trace map $Tr^A_\Lambda: W_3(a^{\geq 1}b) \rightarrow \Lambda^3B$.
We first show the following preliminary result:

\begin{lemma}
\label{lemma:W(ab^4)W(a^3)-GL-invairants}
For any given integer $g\geq 6$, the group $Hom(\overline{W}(ab^4)\otimes \Lambda^3 A,\mathbb{Z})^{GL_g(\mathbb{Z})}$ is generated by the map
$\begin{tikzcd}
	\overline{W}(ab^4)\otimes \Lambda^3 A \ar[r, "{Tr^A_\Lambda\otimes id}"]  & \Lambda^3 B \otimes \Lambda^3 A \ar[r,"{\langle\;,\;\rangle}"] & \mathbb{Z},
	\end{tikzcd}$
	where $\langle\;,\;\rangle$ denotes the perfect pairing between $\Lambda^3 B$ and $\Lambda^3 A$ induced by the symplectic form $\omega$.
\end{lemma}

\begin{proof}
There is an isomorphism
\begin{equation}
\label{eq:hom-W(ab^4)W(a^3)-GL}
Hom(\overline{W}(ab^4)\otimes \Lambda^3 A,\mathbb{Z})^{GL_g(\mathbb{Z})}\simeq Hom((\overline{W}(ab^4)\otimes \Lambda^3 A)_{GL_g(\mathbb{Z})},\mathbb{Z}).
\end{equation}
Consider the following short exact sequence:
\[
\begin{tikzcd}
	0  \ar[r]  &\Gamma_3(W(ab^2)) \ar[r]  & \overline{W}(ab^4) \ar[r] & Q  \ar[r]  & 0,
	\end{tikzcd}
	\]
	where $Q=\overline{W}(ab^4)/\Gamma_3(W(ab^2))$.

Taking the tensor product by $\Lambda^3 A$
and $GL_g(\mathbb{Z})$-coinvariants, we get an exact sequence:
\begin{equation}
\label{eq:ses-T(ab4)-K}
\begin{tikzcd}[column sep=3.5mm]
(\Gamma_3(W(ab^2))\otimes \Lambda^3 A)_{GL_g(\mathbb{Z})} \ar[r]  & (\overline{W}(ab^4)\otimes \Lambda^3 A)_{GL_g(\mathbb{Z})} \ar[r] & (Q\otimes \Lambda^3 A)_{GL_g(\mathbb{Z})}  \ar[r]  & 0.
	\end{tikzcd}
\end{equation}

We prove that the group $(\Gamma_3(W(ab^2))\otimes \Lambda^3 A)_{GL_g(\mathbb{Z})}$ is $3$-torsion.
Consider the $GL_g(\mathbb{Z})$-equivariant epimorphism:
\[
(\otimes ^9H)\otimes(\otimes ^3H) \rightarrow
((W(ab^2)\wedge W(ab^2))\otimes W(ab^2))\otimes \Lambda^3 A \rightarrow
\Gamma_3(W(ab^2))\otimes \Lambda^3 A
\]
Taking $GL_g(\mathbb{Z})$-coinvariants we get another epimorphism:
\[
((\otimes ^9H)\otimes(\otimes ^3H))_{GL_g(\mathbb{Z})} \rightarrow (\Gamma_3(W(ab^2))\otimes \Lambda^3 A)_{GL_g(\mathbb{Z})}.
\]
Then, by Proposition \ref{prop:chords}, this last module is generated by the image of balanced basic tensors of $((\otimes^9 H)\otimes (\otimes^3 H))_{GL_g(\mathbb{Z})}$.
Hence, the module $(\Gamma_3(W(ab^2))\otimes \Lambda^3 A)_{GL_g(\mathbb{Z})}$
is generated by the elements of the form:
\begin{equation}
\label{eq:W(ab^4)W(a^3)-gen-0}
\Bigg[\Bigg[\tree{a_i,b_1,b_2}, \tree{a_j,b_3,b_4} \Bigg], \tree{a_k,b_5,b_6}\Bigg]\otimes \tree{a_l,a_m,a_n}, 
\end{equation}
with pair-wise distinct integers $1\leq i,j,k,l,m,n\leq 6$.
Notice that the first bracket is trivially zero unless there is at least one contraction between the first and second trees. Then without loss of generality we have elements of the form:
\begin{equation}
\label{eq:W(ab^4)W(a^3)-gen}
\Bigg[\Bigg[\tree{a_3,b_1,b_2}, \tree{a_i,b_3,b_4} \Bigg], \tree{a_j,b_5,b_6}\Bigg]\otimes \tree{a_k,a_l,a_m}, 
\end{equation}
with pair-wise distinct integers $1\leq 3,i,j,k,l,m\leq 6$. Next, by all possible contractions of $a_i$ and $a_j$ with the other leaves of trees, we show that these elements are $3$-torsion. It is enough to consider the cases $i=2,4,5$.
\vspace{0.3cm}

$\bullet$ For $i=2$, the element \eqref{eq:W(ab^4)W(a^3)-gen} is zero unless $j=1,4$. Without loss of generality, by the AS relation and the action by the permutation $(1,4)(2,3)$, we can assume that $j=4$. Then we have the element
\begin{equation}
\label{eq:W(ab^4)W(a^3)-gen-1}
\begin{aligned}
\Bigg[\Bigg[\tree{a_3,b_1,b_2}, \tree{a_2,b_3,b_4}\Bigg], \tree{a_4,b_5,b_6}\Bigg] \otimes \tree{a_1,a_5,a_6}
& =
 \Bigg[\tree{b_1,b_2,b_4,a_2}-\tree{a_3,b_1,b_3,b_4}, \tree{a_4,b_5,b_6}\Bigg] \otimes \tree{a_1,a_5,a_6}\\
& = \Bigg(\tree{b_1,b_2,a_2,b_5,b_6}-\tree{b_1,a_3,b_3,b_5,b_6}\Bigg)\otimes \tree{a_1,a_5,a_6}.
\end{aligned}
\end{equation}

Consider the element of $\Gamma_3(W(ab^2))\otimes \Lambda^3 A$,
\[
\Bigg[\Bigg[\tree{b_1,a_2,b_2},\tree{a_2,b_3,b_6}\Bigg],\tree{a_6,b_5,b_6}\Bigg]\otimes \tree{a_1,a_5,a_6}=\tree{b_1,a_2,b_3,b_5,b_6}\otimes \tree{a_1,a_5,a_6}.
\]
This element is zero in the coinvariants module because it is not a balanced element, in particular, the subindex $2$ and $3$ only appear once.
Then, taking the action by $G\in GL_g(\mathbb{Z})$ that sends $a_2$ to $a_2+a_3$ and $b_3$ to $b_3-b_2$, we get that, in the coinvariants module,
\[
\begin{aligned}
0=G\Bigg(\tree{b_1,a_2,b_3,b_5,b_6}\otimes \tree{a_1,a_5,a_6}\Bigg)& =
\Bigg(\tree{b_1,a_2,b_2,b_5,b_6}-\tree{b_1,a_3,b_3,b_5,b_6}\Bigg)\otimes \tree{a_1,a_5,a_6} \\
& =
\Bigg(\tree{b_2,a_2,b_1,b_5,b_6}+\tree{b_1,b_2,a_2,b_5,b_6}-\tree{b_1,a_3,b_3,b_5,b_6}\Bigg)\otimes \tree{a_1,a_5,a_6},
\end{aligned}
\]
where in the last equality we used the IHX relation on the first tree of the first variable.

Adding the opposite of this last element in \eqref{eq:W(ab^4)W(a^3)-gen-1} we get that, in the coinvariants module, the element \eqref{eq:W(ab^4)W(a^3)-gen-1} becomes:
\begin{equation}
\label{eq:W(ab^4)W(a^3)-gen-3}
\tree{a_2,b_2,b_1,b_5,b_6}\otimes \tree{a_1,a_5,a_6}.
\end{equation}
We now prove that this element is $3$-torsion in the coinvariants module. By the IHX relation,
\[
\begin{aligned}
0& = \tree{a_2,b_2,b_1,b_5,b_6}\otimes \tree{a_1,a_5,a_6}-\tree{a_2,b_2,b_5,b_1,b_6}\otimes \tree{a_1,a_5,a_6} -\tree{a_2,b_2,b_6,b_5,b_1}\otimes \tree{a_1,a_5,a_6} \\
& = \tree{a_2,b_2,b_1,b_5,b_6}\otimes \tree{a_1,a_5,a_6}+\tree{a_2,b_2,b_5,b_1,b_6}\otimes \tree{a_5,a_1,a_6} +\tree{a_2,b_2,b_6,b_5,b_1}\otimes \tree{a_6,a_5,a_1}
\end{aligned}
\]
where in the last equality we used the AS relation in the second variable of the last two elements.
Notice that the second and third elements are respectively the result of the action by the permutations $(1,5)$ and $(1,6)$ on the first element.
Therefore, in the coinvariants module, we have that:
\[
3\tree{a_2,b_2,b_1,b_5,b_6}\otimes \tree{a_1,a_5,a_6}=0.
\]

\vspace{0.3cm}

$\bullet$ For $i=4$, the element \eqref{eq:W(ab^4)W(a^3)-gen} is zero unless $j=1,2$. By the AS relation, without loss of generality we can assume $j=2$. 
Then we have the element
\[
\begin{aligned}
\Bigg[\Bigg[\tree{a_3,b_1,b_2}, \tree{a_4,b_3,b_4}\Bigg], \tree{a_2,b_5,b_6}\Bigg] \otimes \tree{a_1,a_5,a_6} & =
\Bigg[\tree{b_1,b_2,b_4,a_4}, \tree{a_2,b_5,b_6}\Bigg] \otimes \tree{a_1,a_5,a_6}
\\
& =-\tree{b_4,a_4,b_1,b_5,b_6} \otimes \tree{a_1,a_5,a_6},
\end{aligned}
\]
which, by the AS relation, is in the $\mathfrak{S}_g$-orbit of \eqref{eq:W(ab^4)W(a^3)-gen-3} and hence it is also $3$-torsion.
\vspace{0.3cm}

$\bullet$ For $i=5$, the element \eqref{eq:W(ab^4)W(a^3)-gen} is zero unless $j=1,2,4,6$. Then we have elements of the following form:
\begin{equation*}
\Bigg[\Bigg[\tree{a_3,b_1,b_2}, \tree{a_5,b_3,b_4}\Bigg], \tree{a_j,b_5,b_6}\Bigg]\otimes \tree{a_k,a_l,a_m}.
\end{equation*}
Notice that for $j=4,6$, by the Jacobi identity, since the first tree of the bracket does not have any contraction with the third tree of the bracket, this element is equal to
\[
\Bigg[\Bigg[ \tree{a_5,b_3,b_4}, \tree{a_j,b_5,b_6}\Bigg], \tree{a_3,b_1,b_2}\Bigg]\otimes \tree{a_k,a_l,a_m},
\]
which is an element of the form given for $i=2,4$.
Therefore we are only left with $j=1,2$ and by the AS relation we can assume $j=2$. In this case we have the element:
\[
\Bigg[\Bigg[\tree{a_3,b_2,b_1}, \tree{a_5,b_3,b_4}\Bigg], \tree{a_2,b_5,b_6}\Bigg]\otimes \tree{a_1,a_4,a_6}.
\]
By the Jacobi identity the following sum of elements is zero:
\[
\begin{array}{c}
\Bigg[\Bigg[\tree{a_3,b_2,b_1}, \tree{a_5,b_3,b_4}\Bigg], \tree{a_2,b_5,b_6}\Bigg]\otimes \tree{a_1,a_4,a_6}  +\Bigg[\Bigg[\tree{a_2,b_5,b_6},\tree{a_3,b_2,b_1}\Bigg], \tree{a_5,b_3,b_4} \Bigg]\otimes \tree{a_1,a_4,a_6} \\
+\Bigg[\Bigg[\tree{a_5,b_3,b_4},\tree{a_2,b_5,b_6}\Bigg], \tree{a_3,b_2,b_1} \Bigg]\otimes \tree{a_1,a_4,a_6}.
\end{array}
\]
Then taking the action by the permutation $(1,4,6)(2,3,5)$ in the second summand and $(1,6,4)(2,5,3)$ in the third summand we get that,
in the coinvariants module, 
\[
3\Bigg[\Bigg[\tree{a_3,b_2,b_1}, \tree{a_5,b_3,b_4}\Bigg], \tree{a_2,b_5,b_6}\Bigg]\otimes \tree{a_1,a_4,a_6}=0.
\]
\vspace{0.3cm}

This shows that the group $(\Gamma_3(W(ab^2))\otimes \Lambda^3 A)_{GL_g(\mathbb{Z})}$ is $3$-torsion.

Therefore, any homomorphism of $Hom((\overline{W}(ab^4)\otimes \Lambda^3 A)_{GL_g(\mathbb{Z})},\mathbb{Z})$ is zero when restricted to $(\Gamma_3(W(ab^2))\otimes \Lambda^3 A)_{GL_g(\mathbb{Z})}$ and hence it factors through the module $(Q\otimes \Lambda^3 A)_{GL_g(\mathbb{Z})}$, where $Q=\overline{W}(ab^4)/\Gamma_3(W(ab^2))$.

We now show that $(Q\otimes \Lambda^3 A)_{GL_g(\mathbb{Z})}$ is generated by a single element.
By Lemma \ref{lema:gen W_1ab4} there is an $GL_g(\mathbb{Z})$-equivariant epimorphism:
\[
\begin{tikzcd}
\otimes^9 H \ar[r]
& (W(a^3)\wedge W(ab^2))\otimes W(a^2b) \ar[r, "{[[\;,\;],\;]}"] & Q.
\end{tikzcd}
\]
This induces another $GL_g(\mathbb{Z})$-equivariant epimorphism,
$(\otimes^9 H)\otimes (\otimes^3 H) \longrightarrow Q \otimes \Lambda^3 A$,
which, after taking $GL_g(\mathbb{Z})$-coinvariants, becomes:
\[
((\otimes^9 H)\otimes (\otimes^3 H))_{GL_g(\mathbb{Z})} \longrightarrow (Q \otimes \Lambda^3 A)_{GL_g(\mathbb{Z})}.
\]
Then, by Proposition \ref{prop:chords}, this last module is generated by the image of balanced basic tensors of $((\otimes^9 H)\otimes (\otimes^3 H))_{GL_g(\mathbb{Z})}$. As a consequence, since $\Lambda^3 A$ is generated by the $\mathfrak{S}_{g}$-orbit of $\tree{a_1,a_2,a_3}$, and, by Lemma \ref{lema:gen W_1ab4}, $Q$ is generated by the $\mathfrak{S}_{g}$-orbit of $\tree{b_1,b_2,b_4,a_4,b_3}$, the module $(Q\otimes \Lambda^3 A)_{GL_g(\mathbb{Z})}$ is generated by the following element:
\[
\tree{b_1,b_2,b_4,a_4,b_3}\otimes \tree{a_1,a_2,a_3}.
\]

Finally, consider the map
\[
\begin{tikzcd}
	\overline{W}(ab^4)\otimes \Lambda^3 A \ar[r, "{Tr^A_\Lambda\otimes id}"]  & \Lambda^3 B \otimes \Lambda^3 A \ar[r,"{\langle\;,\;\rangle}"] & \mathbb{Z}.
	\end{tikzcd}
\]
This map is a well-defined $GL_g(\mathbb{Z})$-invariant homomorphism. Moreover, a direct computation shows that its image on the element $\tree{b_1,b_2,b_4,a_4,b_3}\otimes \tree{a_1,a_2,a_3}$ is $1$. Then the group $Hom(\overline{W}(ab^4)\otimes \Lambda^3 A,\mathbb{Z})^{GL_g(\mathbb{Z})}$ is generated by the map
$\langle Tr^A_\Lambda, id \rangle$.

\end{proof}

\begin{lemma}
\label{lema:Trace-map-A}
For any given integer $g\geq 6$, the subgroup $\tau_3(\mathcal{A}_{g,1}(3))$ is contained in the kernel of the antisymmetric Lagrangian trace map $Tr^A_\Lambda: W_3(a^{\geq 1}b) \rightarrow \Lambda^3 B$.
\end{lemma}

\begin{proof}
To prove the statement we relate the antisymmetric Lagrangian trace $Tr^A_\Lambda$ to the invariant $F=\lambda_2-18\lambda^2+3\lambda$, where $\lambda$ denotes the Casson invariant and $\lambda_2$ the second Ohtsuki invariant.
We first recall some facts about the invariant $F$.
In \cite[Thm. 1.4]{PR} and \cite[Prop 5.1]{PR}, we proved that the invariant $F$ induces a family of functions $F_g$ on $\mathcal{M}_{g,1}(2)$ that are $\mathcal{AB}_{g,1}$-invariant homomorphisms when restricted to $\mathcal{M}_{g,1}(4)$ and vanish on $\mathcal{M}_{g,1}(5)$.
Hence, the map $F_{\mid \mathcal{M}_{g,1}(4)}$ factors through the fourth Johnson homomorphism $\tau_4$, inducing a $GL_g(\mathbb{Z})$-invariant homomorphism
$q: Im(\tau_4)\rightarrow \mathbb{Z}$.

We show that there is a commutative diagram of $\mathcal{AB}_{g,1}$-equivariant maps:
\begin{equation}
\label{eq:com-diag-inv-q}
\begin{tikzcd}
\mathcal{M}_{g,1}(3)\times \mathcal{TA}_{g,1} \ar[r, "{[\;,\;]}"]
\ar[d, "\tau_3\times\tau_1"]  & \mathcal{M}_{g,1}(4)
\ar[d, "\tau_4"] \ar[rd, "F_g"] & \\
	Im(\tau_3)\otimes \Lambda^3 A \ar[r, "{[\;,\;]}"] \ar[d,"\pi\otimes id"]  & Im(\tau_4) \ar[r, "q"] & \mathbb{Z} \\
	\overline{W}_3(a^{\geq 1}b)\otimes \Lambda^3 A
	\ar[r,"{Tr^A_{\Lambda}\otimes id}"] &
	\Lambda^3 B \otimes \Lambda^3 A \ar[ru,"{12\langle \;,\;\rangle}"'] &
	\end{tikzcd}
\end{equation}
where $\pi$ denotes the projection of $Im(\tau_4)$ to $\overline{W}_3(a^{\geq 1}b)$, and $\langle \;,\;\rangle$ is the perfect paring induced by the symplectic form $\omega$. Then we prove that $\tau_3(\mathcal{A}_{g,1}(3))$
is contained in the kernel of the map $Tr^A_\Lambda:W_3(a^{\geq 1}b) \rightarrow \Lambda^3 B$.

We now show the commutativity of diagram \eqref{eq:com-diag-inv-q}.
The commutative triangle in \eqref{eq:com-diag-inv-q} is given by construction of $q$, and the commutative square in \eqref{eq:com-diag-inv-q} is a consequence of the fact that $\tau_k$ is a homomorphism of Lie algebras.
It only remains to show the commutativity of the bottom pentagonal diagram in \eqref{eq:com-diag-inv-q}, ie. $q\circ [\;,\;]=12\langle Tr_\Lambda^A\circ \pi, id\rangle$.

By construction, the maps $q\circ [\;,\;]$ and $\langle Tr_\Lambda^A\circ \pi, id\rangle$ are $GL_g(\mathbb{Z})$-invariant homomorphisms and hence elements of the group
\[
Hom(Im(\tau_3)\otimes \Lambda^3 A,\mathbb{Z})^{GL_g(\mathbb{Z})}
\simeq Hom((Im(\tau_3)\otimes \Lambda^3 A)_{GL_g(\mathbb{Z})},\mathbb{Z}).
\]
We show that this last group is generated by $\langle Tr_\Lambda^A\circ \pi, id\rangle$.

Consider the following composition of $GL_g(\mathbb{Z})$-equivariant epimorphisms:
\[
\begin{tikzcd}
\otimes^9 H \ar[r]
& ((\Lambda^3 H)\wedge(\Lambda^3 H))\otimes (\Lambda^3 H) \ar[r, "{[[\;,\;],\;]}"] & Im(\tau_3),
\end{tikzcd}
\]
where the last surjection is given by the fact that $\tau_1(\mathcal{T}_{g,1})=\Lambda^3 H$, that $\tau_k$ is a Lie algebra homomorphism and $Im(\tau_3)=\tau_3([[\mathcal{T}_{g,1},\mathcal{T}_{g,1}],\mathcal{T}_{g,1}])$ by \cite[Cor. 4.6]{FMS}. This induces another $GL_g(\mathbb{Z})$-equivariant epimorphism,
\[
(\otimes^9 H)\otimes (\otimes^3 H) \longrightarrow Im(\tau_3)\otimes \Lambda^3 A,
\]
which, after taking $GL_g(\mathbb{Z})$-coinvariants, becomes:
\[
((\otimes^9 H)\otimes (\otimes^3 H))_{GL_g(\mathbb{Z})}
\longrightarrow (Im(\tau_3)\otimes \Lambda^3 A)_{GL_g(\mathbb{Z})}.
\]
Then, by Proposition \ref{prop:chords}, this last module is generated by the image of balanced basic tensors of $((\otimes^9 H)\otimes (\otimes^3 H))_{GL_g(\mathbb{Z})}$ and, by Proposition \ref{prop:desc-m(3)}, we get an isomorphism:
\[
(Im(\tau_3)\otimes \Lambda^3 A)_{GL_g(\mathbb{Z})}\simeq (\overline{W}(ab^4)\otimes \Lambda^3 A)_{GL_g(\mathbb{Z})}.
\]
Hence,
\[
Hom((Im(\tau_3)\otimes \Lambda^3 A)_{GL_g(\mathbb{Z})},\mathbb{Z})\simeq
Hom((\overline{W}(ab^4)\otimes \Lambda^3 A)_{GL_g(\mathbb{Z})},\mathbb{Z}),
\]
where, by Lemma \ref{lemma:W(ab^4)W(a^3)-GL-invairants}, the first group is generated by $\langle Tr_\Lambda^A\circ \pi, id\rangle$.
Therefore, there is an integer $k$ such that $q\circ [\;,\;]=k\langle Tr_\Lambda^A\circ \pi, id\rangle$.

We now show that $k=12$, by computing the image of the maps $q\circ [\;,\;]$ and $\langle Tr_\Lambda^A\circ \pi, id\rangle$ on the element
$T_1\otimes T_2\in Im(\tau_3)\otimes \Lambda^3 A$, where,
\[
T_1=\Bigg[\Bigg[\tree{b_2,b_1,b_5},\tree{a_5,a_1,b_6}\Bigg],\tree{a_6,b_3,b_4}\Bigg] \quad
\text{and} \quad T_2=-\tree{a_4,a_3,a_2}.
\]

We first compute $q([T_1, T_2])$.
Since $\tau_k$ is a Lie algebra homomorphism and it is surjective for $k=1$, there exist $\varphi\in \mathcal{T}_{g,1}(3)$ and $\xi\in \mathcal{TA}_{g,1}$ such that $\tau_3(\varphi)=T_1$ and $\tau_1(\xi)=T_2$.
Then, by commutative diagram \eqref{eq:com-diag-inv-q} and the Jacobi identity,
\[
\begin{aligned}
{[}T_1,T_2]=\tau_4([\varphi, \xi])
= &
-\Bigg[\Bigg[\Bigg[\tree{b_2,b_1,b_5},\tree{a_5,a_1,b_6}\Bigg],\tree{a_6,b_3,b_4}\Bigg], \tree{a_4,a_3,a_2}\Bigg] \\
= & \; \Bigg[\Bigg[\tree{a_6,b_3,b_4},\tree{a_4,a_3,a_2}\Bigg], \Bigg[\tree{b_2,b_1,b_5},\tree{a_5,a_1,b_6}\Bigg]\Bigg] \\
& +\Bigg[\Bigg[\tree{a_4,a_3,a_2},\Bigg[\tree{b_2,b_1,b_5},\tree{a_5,a_1,b_6}\Bigg]\Bigg],\tree{a_6,b_3,b_4}\Bigg] \\
= & \; \Bigg[\Bigg[\tree{a_6,b_3,b_4},\tree{a_4,a_3,a_2}\Bigg], \Bigg[\tree{b_2,b_1,b_5},\tree{a_5,a_1,b_6}\Bigg]\Bigg] \\
& -\Bigg[\Bigg[\tree{a_4,a_3,b_2},\Bigg[\tree{a_2,b_1,b_5},\tree{a_5,a_1,b_6}\Bigg]\Bigg],\tree{a_6,b_3,b_4}\Bigg].
\end{aligned}
\]
where in the last equality we exchanged $a_2$ and $b_2$ in the second summand.
Notice that the second summand in the last equality belongs to 
$\Gamma_4(W(a^{\geq 1}b^{\geq 1}))\subseteq\tau_4(\mathcal{TAB}_{g,1}(4))$,
where this last inclusion comes from the fact that $\tau_k$ is a Lie albegra homomorphism and $\tau_1(\mathcal{TAB}_{g,1})=W(a^{\geq 1}b^{\geq 1})$, by \cite[Prop. 6.2]{faes1}.
Then, since $F_g$ is an invariant of integral homology spheres, $F_g$ it is zero on $\mathcal{TAB}_{g,1}(4)$ and, by commutative diagram \eqref{eq:com-diag-inv-q}, the map $q$ is zero on the second summand. Hence,
\[
q([T_1,T_2])=q\Bigg(\Bigg[\Bigg[\tree{a_6,b_3,b_4},\tree{a_4,a_3,a_2}\Bigg], \Bigg[\tree{b_2,b_1,b_5},\tree{a_5,a_1,b_6}\Bigg]\Bigg]\Bigg).
\]
Since $\tau_k$ is a Lie algebra homomorphism and, by \cite[Lem. 2.5]{mor},
$\tau_1(\mathcal{TA}_{g,1}) =W(a^{\geq 1}b)$ and 
$\tau_1(\mathcal{TB}_{g,1}) =W(ab^{\geq 1})$,
then there exist $\psi_1\in \mathcal{TA}_{g,1}(2)$ and $\psi_2\in \mathcal{TB}_{g,1}(2)$ such that
\[
\tau_2(\psi_1)=\Bigg[\tree{a_6,b_3,b_4},\tree{a_4,a_3,a_2}\Bigg] \quad 
\text{and} \quad
\tau_2(\psi_2)=\Bigg[\tree{b_2,b_1,b_5},\tree{a_5,a_1,b_6}\Bigg].
\]
Therefore we have the following equalities:
\begin{equation}
\label{eq:image-q}
q([T_1,T_2])=q([\tau_2(\psi_1),\tau_2(\psi_2)]) = q\circ \tau_4([\psi_1,\psi_2])=
F_g([\psi_1,\psi_2]).
\end{equation}

To compute $F_g([\psi_1,\psi_2])$ we use the $2$-cocycle $C_g$ associated to the invariant $F_g$, that is,
\[
C_g(\phi,\psi)=F_g(\phi)+F_g(\psi)-F_g(\phi\psi),\qquad \forall\;\phi,\psi\in \mathcal{M}_{g,1}(2).
\]
By \cite[Sec. 3]{PR} this $2$-cocycle is $\mathcal{AB}_{g,1}$-invariant and zero on 
$\mathcal{A}_{g,1}(2)\times \mathcal{M}_{g,1}(2)$ and
$\mathcal{M}_{g,1}(2)\times \mathcal{B}_{g,1}(2)$. Moreover, by \cite[Sec. 6]{PR}, this 2-cocycle is the 
pull-back along the second Johnson homomorphism $\tau_2$ of a
bilinear form $B_g$ on $\mathcal{A}_{2}(H)$, which by \cite[Prop. 6.3]{PR}, is given by:
\[
B_g= 48J_g+ 12Q_g,
\]
where $J_g$ and $Q_g$ are $GL_g(\mathbb{Z})$-invariant bilinear forms on $\mathcal{A}_2(H)$ that respectively factor through $W(b^4)\otimes W(a^4)$ and $W(ab^3)\otimes W(a^3b)$. By direct computation,
\[
J_g\Bigg(\tree{b_1,b_2,b_3,b_4}, \tree{a_1,a_2,a_3,a_4}\Bigg)=1\quad \text{and}
\quad Q_g\Bigg(\tree{b_1,a_3,b_3,b_2}, \tree{a_1,a_4,b_4,a_2}\Bigg)=1.
\]

Then, we compute that
\begin{equation*}
\begin{aligned}
F_g([\psi_1,\psi_2]) & = F_g(\psi_1\psi_2)-F_g(\psi_2\psi_1)-C_g(\psi_1\psi_2,[\psi_1,\psi_2]) \\
& = C_g(\psi_1,\psi_2)-C_g(\psi_2,\psi_1)\\
& = -C_g(\psi_2,\psi_1),
\end{aligned}
\end{equation*}
where,
\[
C_g(\psi_2,\psi_1)=B_g(\tau_2(\psi_2),\tau_2(\psi_1))=B_g\Bigg(\tree{b_2,b_1,a_1,b_6}+\tree{b_2,b_5,a_5,b_6},\tree{a_6,b_3,a_3,a_2}+\tree{a_6,b_4,a_4,a_2}\Bigg)= 48.
\]
Hence, by \eqref{eq:image-q},
\begin{equation}
\label{eq:left-map}
q([T_1, T_2])=-48.
\end{equation}

We now compute $\langle Tr_\Lambda^A(\pi(T_1)), T_2\rangle$.
By construction we have that:
\[
\begin{aligned}
T_1=\Bigg[\Bigg[\tree{b_2,b_1,b_5},\tree{a_5,a_1,b_6}\Bigg],\tree{a_6,b_3,b_4}\Bigg]
& =\Bigg[-\tree{b_2,b_1,a_1,b_6}-\tree{b_2,b_5,a_5,b_6},\tree{a_6,b_3,b_4}\Bigg] \\
& = \tree{b_2,b_1,a_1,b_3,b_4}+\tree{b_2,b_5,a_5,b_3,b_4}.
\end{aligned}
\]
Then we compute that $Tr_\Lambda^A(\pi(T_1))= 4\tree{b_2,b_3,b_4}$ and hence,
\begin{equation}
\label{eq:right-map}
\langle Tr_\Lambda^A(\pi(T_1)), T_2\rangle=-4\Bigg\langle \tree{b_2,b_3,b_4}, \tree{a_4,a_3,a_2}\Bigg\rangle=-4.
\end{equation}

Therefore, from \eqref{eq:left-map} and \eqref{eq:right-map} we get that $k=12$ and hence,
\[
q\circ [\;,\;]=12\langle Tr_\Lambda^A\circ \pi, id\rangle.
\]
This finishes the proof of the commutativity of diagram \eqref{eq:com-diag-inv-q}.

Finally, we show that $\tau_3(\mathcal{A}_{g,1}(3))$
is contained in the kernel of the antisymmetic Lagrangian trace map $Tr^A_\Lambda:W_3(a^{\geq 1}b) \rightarrow \Lambda^3 B$.

By \cite[Sec. 4]{faes1}, we have an inclusion
\[
\tau_3(\mathcal{A}_{g,1}(3))\subseteq \tau_3(\mathcal{M}_{g,1}(3)) \cap
Ker(D_3(H)\rightarrow D_3(H/A)).
\]
Hence,
$\tau_3(\mathcal{A}_{g,1}(3))\subseteq W_3(a^{\geq 1}b)$.
We now prove that, indeed, $\tau_3(\mathcal{A}_{g,1}(3))\subseteq Ker(Tr^A_\Lambda)$.

Given an element $S\in \tau_3(\mathcal{A}_{g,1}(3))$, such that
\[Tr_\Lambda^A(S)=\sum_{1\leq i<j<k \leq g}\lambda_{ijk}\tree{b_i,b_j,b_k},
\quad \text{with} \quad \lambda_{ijk}\in \mathbb{Z}.
\]
We show that $\lambda_{ijk}=0$ for all integers
$1\leq i<j<k \leq g$.

For each triple of integers $1\leq i'<j'<k'\leq g$ consider the element $T=\tree{a_{i'},a_{j'},a_{k'}}\in \Lambda^3 A$.
Since $S\in \tau_3(\mathcal{A}_{g,1}(3))$ and $\tau_1(\mathcal{TA}_{g,1})=\Lambda^3 A$, there exist
$\varphi \in \mathcal{A}_{g,1}(3)$ and $\psi \in \mathcal{TA}_{g,1}$ such that $\tau_3(\varphi)=S$ and $\tau_1(\psi)=T$. Then $[\varphi, \psi]\in \mathcal{A}_{g,1}(4)$
and, since the map $F_g$ is an invariant of integral homology spheres, $F_g$ is zero on this bracket:
\[
0=F_g([\varphi, \psi]) = q([\tau_3(\phi),\tau_1(\psi)])=q([S,T]).
\]
This equality and the commutative diagram \eqref{eq:com-diag-inv-q} gives us the following equalities:
\[
0=
12\langle Tr_\Lambda^A(\pi(S)),T \rangle =  12\langle Tr_\Lambda^A(S),T \rangle  =  12\sum_{i<j<k}\lambda_{ijk}\Bigg\langle \tree{b_i,b_j,b_k},\tree{a_{i'},a_{j'},a_{k'}}\Bigg\rangle  =  -12\lambda_{i'j'k'}.
\]
Therefore $\lambda_{i'j'k'}=0$ for all triples $i'<j'<k'$ and hence $Tr^A_\Lambda(S)=0$.
\end{proof}

We compute the image and the kernel of the antisymmetric Lagrangian trace map $Tr^A_\Lambda$.

\begin{lemma}
\label{lema:ses-Trace-map-A}
For any integer $g\geq 6$, the trace map $Tr^A_\Lambda$
induces a short exact sequence:
\[
\begin{tikzcd}
0 \ar[r] & \overline{W}_3(a^{\geq 2}b)\oplus \Gamma_3(W(ab^2)) \ar[r]  & \overline{W}_3(a^{\geq 1}b) \ar[r,"{Tr^A_\Lambda}"] & 2\Lambda^3 B \ar[r] & 0.
\end{tikzcd}
\]
\end{lemma}

\begin{proof}
By Lemma \ref{lema:gen W_1ab4} and Table \ref{table:1}, $\overline{W}_3(a^{\geq 1}b)=\overline{W}_3(a^{\geq 2}b)\oplus \Gamma_3(W(ab^2))+ K$, where $K$ is the subgroup generated by the $\mathfrak{S}_g$-orbits of the tree $\tree{b_1,b_2,b_4,a_4,b_3}$.

By definition of the trace map $Tr^A_\Lambda$, the submodule $\overline{W}_3(a^{\geq 2}b)$ is in its kernel.
We know that $\Gamma_3(W(ab^2))\subseteq \tau_3(\mathcal{TAB}_{g,1}(3))\subseteq \tau_3(\mathcal{TA}_{g,1}(3))$ and by Lemma \ref{lema:Trace-map-A} the trace map $Tr^A_\Lambda$ is zero on these elements.
Therefore it only remains to compute the kernel and the image of the trace map $Tr^A_\Lambda$ on $K$.
For the image, we have
\[
Tr^A_\Lambda\Bigg(\tree{b_k,a_l,b_l,b_j,b_i}\Bigg)=2\tree{b_i,b_j,b_k}.\]
Hence the image of $Tr^A_\Lambda$ is $2\Lambda^3 B$.
We now show that the kernel of the trace map $Tr^A_\Lambda$ on $K$ is contained in $\Gamma_3(W(ab^2))$.
By the AS relation, an element of $K$ can be written as:
\[
S=\sum_{\substack{i<j<k \\ l\neq i,j,k}}\lambda_{ijkl}\tree{b_i,b_j,b_l,a_l,b_k}+\lambda_{jkil}\tree{b_j,b_k,b_l,a_l,b_i}.
\]
Assume that this element belongs to the kernel of $Tr^A_\Lambda$. Then we compute:
\[
0=Tr^A_\Lambda(S)=\sum_{\substack{i<j<k \\ l\neq i,j,k}}2\lambda_{ijkl}\tree{b_i,b_j,b_k}+2\lambda_{jkil}\tree{b_j,b_k,b_i}=2\sum_{\substack{i<j<k \\ l\neq i,j,k}}(\lambda_{ijkl}+\lambda_{jkil})\tree{b_i,b_j,b_k}.
\]
Hence, for each fixed triple $i<j<k$,
$
\sum_{\substack{l\neq i,j,k}}(\lambda_{ijkl}+\lambda_{jkil})=0
$,
and we have that:
\[
\begin{aligned}
S & =S-\sum_{\substack{i<j<k \\ l\neq i,j,k}}(\lambda_{ijkl}+\lambda_{jkil})\tree{b_i,b_j,b_k,a_k,b_k} \\
& = \sum_{i<j<k}\Bigg(\sum_{l\neq i,j,k}
\lambda_{ijkl}\Bigg(\tree{b_i,b_j,b_l,a_l,b_k}
-\tree{b_i,b_j,b_k,a_k,b_k}\Bigg)
+\lambda_{jkil}\Bigg(\tree{b_j,b_k,b_l,a_l,b_i}-\tree{b_i,b_j,b_k,a_k,b_k}\Bigg)\Bigg).
\end{aligned}
\]
where the last difference of trees can be written as
\[
\tree{b_j,b_k,b_l,a_l,b_i}-\tree{b_i,b_j,b_k,a_k,b_k}=
\tree{b_j,b_k,b_l,a_l,b_i}-\tree{b_i,b_j,b_l,a_l,b_k}
+\tree{b_i,b_j,b_l,a_l,b_k}-\tree{b_i,b_j,b_k,a_k,b_k}.
\]
Finally, by points $iii)$, $v)$ in Lemma \ref{lema:trees-ab4-needed} and the AS relation, this sum of trees belongs indeed to $\Gamma_3(W(ab^2))$.
\end{proof}

\subsection{The image of the handlebody subgroups}

We now compute the image of the third Johnson homomorphism on the handlebody subgroups $\mathcal{A}_{g,1}(3)$, $\mathcal{B}_{g,1}(3)$ and $\mathcal{AB}_{g,1}(3)$.

\begin{proposition}
\label{prop:Handlebody A(3)}
For any integer $g\geq 6$ we have that,
\[
\begin{aligned}
\tau_3(\mathcal{A}_{g,1}(3)) & =  \tau_3(\mathcal{TA}_{g,1}(3))=Ker(Tr^A_\Lambda)\cap Im(\tau_3), \\
\tau_3(\mathcal{B}_{g,1}(3)) & =  \tau_3(\mathcal{TB}_{g,1}(3))=Ker(Tr^B_\Lambda)\cap Im(\tau_3), \\
\tau_3(\mathcal{AB}_{g,1}(3)) & =  \tau_3(\mathcal{TAB}_{g,1}(3))=Ker(Tr^A_\Lambda)\cap Ker(Tr^B_\Lambda)\cap Im(\tau_3).
\end{aligned}
\]
\end{proposition}

\begin{proof}

We first show that:
\begin{equation}
\label{eq:inclusions}
\begin{array}{c}
\tau_3(\mathcal{TA}_{g,1}(3))  \supseteq Ker(Tr^A_\Lambda)\cap Im(\tau_3),\qquad
\tau_3(\mathcal{TB}_{g,1}(3))  \supseteq Ker(Tr^B_\Lambda)\cap Im(\tau_3),\\[1ex]
\tau_3(\mathcal{TAB}_{g,1}(3)) \supseteq Ker(Tr^A_\Lambda)\cap Ker(Tr^B_\Lambda)\cap Im(\tau_3).
\end{array}
\end{equation}

By Lemma \ref{lema:ses-Trace-map-A} we have the following equalities:
\[
\begin{aligned}
Ker(Tr^A_\Lambda)\cap Im(\tau_3)& = \overline{W}_3(a^{\geq 2}b) \oplus \Gamma_3(W(ab^2)), \\
Ker(Tr^B_\Lambda)\cap Im(\tau_3)& =\Gamma_3(W(a^2b))\oplus \overline{W}_3(ab^{\geq 2}), \\
Ker(Tr^A_\Lambda)\cap Ker(Tr^B_\Lambda)\cap Im(\tau_3)& =\Gamma_3(W(a^2b))\oplus \overline{W}_3(a^{\geq 2}b^{\geq 2})\oplus \Gamma_3(W(ab^2)).
\end{aligned}
\]

Since the Johnson homomorphism is a Lie algebra homomorphism, we have that
\[
\begin{aligned}
\tau_3(\mathcal{TA}_{g,1}(3))& =[[\tau_1(\mathcal{TA}_{g,1}),\tau_1(\mathcal{TA}_{g,1})],\tau_1(\mathcal{TA}_{g,1})], \\
\tau_3(\mathcal{TB}_{g,1}(3))& =[[\tau_1(\mathcal{TB}_{g,1}),\tau_1(\mathcal{TB}_{g,1})],\tau_1(\mathcal{TB}_{g,1})], \\
\tau_3(\mathcal{TAB}_{g,1}(3))& =[[\tau_1(\mathcal{TAB}_{g,1}),\tau_1(\mathcal{TAB}_{g,1})],\tau_1(\mathcal{TAB}_{g,1})].
\end{aligned}
\]
Moreover, by \cite[Lem. 2.5]{mor} and \cite[Cor. 6.5]{faes1}, we know that
\[
\tau_1(\mathcal{TA}_{g,1})  =W(a^{\geq 1}b), \qquad \tau_1(\mathcal{TB}_{g,1})  =W(ab^{\geq 1}) \quad \text{and} \quad
\tau_1(\mathcal{TAB}_{g,1}) =W(a^{\geq 1}b^{\geq 1}). 
\]
Then by Table \ref{table:1} we get the inclusions \eqref{eq:inclusions}, and by Lemma \ref{lema:Trace-map-A} we have that:
\[
\begin{aligned}
& \tau_3(\mathcal{A}_{g,1}(3))\subseteq Im(\tau_3)\cap Ker(Tr^A_\Lambda) \subseteq \tau_3(\mathcal{TA}_{g,1}(3)) \subseteq \tau_3(\mathcal{A}_{g,1}(3)), \\
& \tau_3(\mathcal{B}_{g,1}(3))\subseteq Im(\tau_3)\cap Ker(Tr^B_\Lambda) \subseteq \tau_3(\mathcal{TB}_{g,1}(3)) \subseteq \tau_3(\mathcal{B}_{g,1}(3)), \\
& \tau_3(\mathcal{AB}_{g,1}(3))\subseteq Im(\tau_3)\cap Ker(Tr^A_\Lambda)\cap Ker(Tr^B_\Lambda) \subseteq \tau_3(\mathcal{TAB}_{g,1}(3)) \subseteq \tau_3(\mathcal{AB}_{g,1}(3)).
\end{aligned}
\]
Therefore, all these inclusions are indeed equalities.
\end{proof}

As a first consequence of Proposition \ref{prop:Handlebody A(3)} we get a  handlebody version of \cite[Cor. 4.6]{FMS}.

\begin{corollary}
\label{cor:Handlebody A(3)}
For any integer $g\geq 6$ we have that,
\[
\begin{array}{c}
\mathcal{A}_{g,1}(3)=\mathcal{TA}_{g,1}(3)\cdot\mathcal{A}_{g,1}(4),\qquad
\mathcal{B}_{g,1}(3)=\mathcal{TB}_{g,1}(3)\cdot\mathcal{B}_{g,1}(4), \\[1ex]
\mathcal{AB}_{g,1}(3)=\mathcal{TAB}_{g,1}(3)\cdot\mathcal{AB}_{g,1}(4).
\end{array}
\]
\end{corollary}

As a second consequence of Proposition \ref{prop:Handlebody A(3)} we compute the intersection of the images of $\mathcal{A}_{g,1}(3)$ and $\mathcal{B}_{g,1}(3)$ by the third Johnson homomorphism.

\begin{corollary}
\label{cor:A(3)-cap-B(3)}
For any integer $g\geq 6$ we have that,
\[\tau_3(\mathcal{AB}_{g,1}(3))=\tau_3(\mathcal{A}_{g,1}(3))\cap \tau_3(\mathcal{B}_{g,1}(3)).\]
\end{corollary}

\subsection{Equivalence relation for the 4th Johnson subgroup}

Once we have computed the images of the handlebody subgroups, we finally prove the core Theorem of this work.

\begin{theorem}
The Heegaard map induces a bijection:
\[
\lim_{g\to \infty} (\mathcal{A}_{g,1}(4)\backslash \mathcal{M}_{g,1}(4) / \mathcal{B}_{g,1}(4))_{\mathcal{AB}_{g,1}} \simeq \mathcal{S}_\mathbb{Z}^3,
\]
More precisely, two maps $\phi,\psi \in \mathcal{M}_{g,1}(4)$ are equivalent if and only if there exist maps $\xi_a \in \mathcal{A}_{g,1}(4)$, $\xi_b \in \mathcal{B}_{g,1}(4)$ and $\mu \in \mathcal{AB}_{g,1}$ such that
$\phi = \mu \xi_a \psi \xi_b \mu^{-1}$.
\end{theorem}

\begin{proof}
By \cite[Thm. B]{faes2} we already know that this map is surjective and well-defined. We prove that it is injective. Consider $\phi, \psi\in \mathcal{M}_{g,1}(4)$ with image the same homology sphere. By \cite[Prop 6.7]{faes1} we know that possibly after stabilization, there exist elements $\xi_a\in \mathcal{A}_{g,1}(3)$, $\xi_b\in \mathcal{B}_{g,1}(3)$ and $\mu\in \mathcal{AB}_{g,1}$ such that 
$\mu\xi_a \varphi \xi_b \mu^{-1}=\psi$. Then applying $\tau_3$ to this equality we get that $\tau_3(\xi_a)+\tau_3(\xi_b)=0$ and hence $\tau_3(\xi_a)=-\tau_3(\xi_b)$
Therefore $\tau_3(\xi_a)\in \tau_3(\mathcal{A}_{g,1}(3))\cap \tau_3(\mathcal{B}_{g,1}(3))$ and by Corollary \ref{cor:A(3)-cap-B(3)} there exist an element $\nu\in \mathcal{AB}_{g,1}(3)$ such that
$\tau_3(\nu)=\tau_3(\xi_a)=-\tau_3(\xi_b)$. As a consequence,
\[\psi= \mu\xi_a \varphi \xi_b \mu^{-1}=\mu\nu\nu^{-1}\xi_a \varphi \xi_b \nu\nu^{-1} \mu^{-1}=(\mu\nu)(\nu^{-1}\xi_a) \varphi (\xi_b \nu)(\mu\nu)^{-1},\]
where $\mu\nu \in \mathcal{AB}_{g,1}$, $\nu^{-1}\xi_a\in \mathcal{A}_{g,1}(4)$ and $\xi_b \nu \in \mathcal{B}_{g,1}(4)$ because $\tau_3(\nu^{-1}\xi_a)=0$,
$\tau_3(\xi_b \nu)=0$.

\end{proof}

\pagebreak

\bibliography{biblio}{}
\bibliographystyle{abbrv}

\end{document}